\UseRawInputEncoding
\documentclass[11pt]{article}
\usepackage{amsfonts}
\usepackage{graphics}
\usepackage{indentfirst}
\usepackage{cite}
\usepackage{latexsym}
\usepackage{amsmath}
\usepackage{amssymb}
\usepackage[dvips]{epsfig}
\usepackage{amscd}
\usepackage{mathtools}
\usepackage[bookmarks,colorlinks,citecolor=blue,linkcolor=blue]{hyperref}
\clearpage
\allowdisplaybreaks
\newtheorem{theorem}{Theorem}[section]
\newtheorem{remark}{Remark}[section]

\newtheorem{lemma}[theorem]{Lemma}

\newtheorem{proposition}[theorem]{Proposition}

\renewcommand{\div}{ {\rm div }  }

\newcommand{\bt}{\begin{theorem}}
\newcommand{\bl}{\begin{lemma}}
\newcommand{\el}{\end{lemma}}
\newcommand{\et}{\end{theorem}}

\newcommand{\curl}{{\rm curl} }

\newcommand{\bn}{\begin{eqnarray}}
\newcommand{\en}{\end{eqnarray}}
\newcommand{\bnn}{\begin{eqnarray*}}
\newcommand{\enn}{\end{eqnarray*}}

\newcommand{\bnnn}{\begin{eqnarray*}}
\newcommand{\ennn}{\end{eqnarray*}}

\newcommand{\ba}{\begin{aligned}}
\newcommand{\ea}{\end{aligned}}
\newcommand{\be}{\begin{equation}}
\newcommand{\ee}{\end{equation}}

\def\norm[#1]#2{\|#2\|_{#1}}

\usepackage[marginal]{footmisc}
\renewcommand{\thefootnote}{}
\newcommand\blfootnote[1]{%
  \begingroup
  \renewcommand\thefootnote{}\footnote{#1}%
  \addtocounter{footnote}{-1}%
  \endgroup
}

\makeatletter      
\@addtoreset{equation}{section}
\makeatother       
\title{Global Strong Solutions to the incompressible Magnetohydrodynamic Equations with Density-Dependent Viscosity and Vacuum in 3D Exterior Domains}
\author{$\text{Bing Yuan}^a, \text{Rong Zhang}^{a,b}, \text{Peng Zhou}^{a,c \ *}$\\
a. School of Mathematics and Computer Sciences,\\
Nanchang University, Nanchang 330031, P. R. China;\\
b. Institute of Mathematics and Interdisciplinary Sciences,\\
Nanchang University, Nanchang 330031, P. R. China;\\
c. Jiangxi Institute Economic Administrators,\\
Nanchang 330088, P. R. China}
\date{}
\begin{document}
\maketitle
\blfootnote{*Email: 2541417731@qq.com(B. Yuan), rzhang0921@gmail.com(R. Zhang), 975405097@qq.com(P. Zhou)}
\begin{abstract}
The nonhomogeneous incompressible Magnetohydrodynamic Equations with density-dependent viscosity is studied in three-dimensional (3D) exterior domains with slip boundary conditions. The key is the constraint of an additional initial value condition $B_0\in L^p (1\leqslant p<12/7)$, which increase decay-in-time rates of the solutions, thus we obtain the global existence of strong solutions provided the gradient of the initial velocity and initial magnetic field is suitably small.  In particular, the initial density is allowed to contain vacuum states and large oscillations. Moreover, the large-time behavior of the solution is also shown.
\end{abstract}
\section{Introduction}
The nonhomogeneous incompressible magnetohydrodynamic (MHD) equations with density-dependent viscosity read as follows:
\begin{equation}\label{MHD}
\begin{cases}
\rho_t+\div(\rho u)=0, \\
(\rho u)_t+\div(\rho u \otimes u)-\div(2 \mu(\rho) d)+\nabla P=B \cdot \nabla B, \\
B_t+u \cdot \nabla B-B \cdot \nabla u-\nu \Delta B=0, \\
\div u=0, \quad \div B=0.
\end{cases}
\end{equation}
Here, $t \geqslant 0$ is time, $x \in \Omega \subset \mathbb{R}^3$ is the spatial coordinate, and the unknown functions $\rho=\rho(x, t), u=(u^1, u^2, u^3)(x, t), B=$ $(B^1, B^2, B^3)(x, t)$ and $P=P(x, t)$ denote the fluid density, velocity, magnetic field and pressure, respectively. The deformation tensor is defined by
\begin{equation}\label{ddyi}
\displaystyle d=\frac{1}{2}\left[\nabla u+(\nabla u)^T\right],
\end{equation}
the viscosity $\mu(\rho)$ satisfies the following hypothesis:
\begin{equation}\label{vis}
\displaystyle \mu \in C^1[0, \infty), \quad \mu(\rho)>0 ,
\end{equation}
the constant $\nu>0$ is the resistivity coefficient which is inversely proportional to the electrical conductivity constant and acts as the magnetic diffusivity of magnetic fields.

In this paper, our purpose is to study the global existence and large-time asymptotic behavior of strong solutions of $\eqref{MHD}$ in the exterior of a simply connected bounded domain in $\mathbb{R}^3$. More precisely, the system $\eqref{MHD}$ will be investigated under the assumptions: The domain $\Omega$ is the exterior of a simply connected bounded smooth domain $D$ in $\mathbb{R}^3$ which represents the obstacle, i.e. $\Omega=\mathbb{R}^3-\bar{D}$ and its boundary $\partial \Omega$ is smooth. The system $\eqref{MHD}$ is equipped with the given initial data
\begin{equation}\label{ini}
\displaystyle (\rho, \rho u, B)|_{t=0}=(\rho_0,m_0,B_0), \quad \text{ in }  \Omega .
\end{equation}
In order to close the system, we need some boundary conditions and a condition at infinity. In this paper, we assume that the following boundary conditions
\begin{equation}\label{ubj}
\displaystyle u\cdot n=0, \ (d\cdot n)_{tan}=0,\quad \text { on } \partial \Omega,
\end{equation}
\begin{equation}\label{bbj}
\displaystyle B\cdot n=0, \ \curl B\times n=0,\quad \text { on } \partial \Omega,
\end{equation}
with the far field behavior
\begin{equation}\label{ydxw}
\displaystyle (u,B) \rightarrow (0,0),\quad \text { as }|x| \rightarrow \infty.
\end{equation}
It is worth noting that $\eqref{ubj}$ is a kind of slip boundary condition, where the symbol $v_{\text {tan}}$ represents the projection of tangent plane of the vector $v$ on $\partial \Omega$. For the magnetic field, the boundary condition $\eqref{bbj}$ describes that the boundary $\partial\Omega$ is a perfect conductor (see e.g., \cite{Dou2013}).

Magnetohydrodynamics studies the dynamics of electrically conducting fluids and the theory of the macroscopic interaction of electrically conducting fluids with a magnetic field. The dynamic motion of the fluid and the magnetic field interact strongly with each other, so the hydrodynamic and electrodynamic effects are coupled.
If there is no electromagnetic effect, that is H = 0, the MHD system reduces to the Navier-Stokes equations, which have been discussed by many mathematicians, please see \cite{Cho2004,Cho2003,He2021,Hua2015,Hua2014,Hua2012,Lio1996,Lij1673,Lij2014,Lia2015,Lub3143,Zha2015} and references therein.
Now, we briefly recall some results concerned with well-posedness of solutions for multi-dimensional incompressible MHD equations which are more relatively with our problem. Duvaut-Lions \cite{Duv1972} established the existence of the global weak solution to the 3D homogeneous MHD equations with finite energy, and also showed the smoothness and uniqueness in the two dimensional cased provided given initial data are smooth. The local strong solutions and the global ones with small initial data were constructed by Sermange-Teman \cite{Ser1983}. Gerbeau-Le Bris \cite{Ger1977} and Desjardins-Le Bris \cite{Des1998} established the global existence of weak solutions to the 3D nonhomogeneous MHD equations with finite energy on 3D bounded domains and on the torus, respectively. Abidi-Hmidi \cite{Abi2007} and Abidi-Paicu \cite{Abi2008} established the local and global existence of strong solutions in some Besov spaces when the initial density $\rho_0$ is bounded away from zero and the initial data is suitable small.

Recently, L\"u-Xu-Zhong \cite{Lub2017} established the global existence and uniqueness of strong solutions to the 2D Cauchy problem with initial data can be arbitrarily large and the contain vacuum state. Song \cite{Son2018} obtained the local strong solutions to the 3D nonhomogeneous incompressible MHD equations with density-dependent viscosity and vacuum. Zhang \cite{Zha2020} established the global existence and uniqueness of strong solutions to the 3D Cauchy problem provided that the initial velocity and magnetic field are suitable small in some nonhomogeneous Sobolev spaces. Zhang-Zhao \cite{Zha2010} obtained optimal decay estimates of the solution when the initial perturbation is sufficiently small and bounded in some $L^p$. Li \cite{Lih2018} established the global existence and uniqueness of strong solutions provided the initial velocity and current density are both suitably small in a 3D bounded domain with slip boundary condition. However, there are no works about the global existence of the strong (classical) solution to the initial-boundary-value problem $\eqref{MHD}$-$\eqref{ydxw}$ in a 3D exterior domain with initial density containing vacuum, at least to the best of our knowledge. In fact, this is the main aim of this paper.


Before stating the main results, we first explain the notations and conventions used throughout this paper. Set
$$
\int f \mathrm{~d} x \triangleq \int_{\Omega} f \mathrm{~d} x \text {. }
$$
Next, for $1 \leq r \leq \infty, k \geq 1$, the standard homogeneous Sobolev spaces are defined as follows:
$$
\left\{\begin{array}{l}
L^r=L^r(\Omega), \quad W^{k, r}=W^{k, r}(\Omega), \quad H^k=W^{k, 2}, \\
\|\cdot\|_{B_1 \cap B_2}=\|\cdot\|_{B_1}+\|\cdot\|_{B_2}, \text { for two Banach spaces } B_1 \text { and } B_2, \\
D^{k, r}=D^{k, r}(\Omega)=\{v \in L_{\text {loc }}^1(\Omega) \mid \nabla^k v \in L^r(\Omega)\}, \\
D^1=\{v \in L^6(\Omega) \mid \nabla v \in L^2(\Omega)\}, \\
C_{0, \sigma}^{\infty}=\{f \in C_0^{\infty} \mid \operatorname{div} f=0\}, \quad D_{0, \sigma}^1=\overline{C_{0, \sigma}^{\infty}} \text { closure in the norm of } D^1, \\
H_{0, \sigma}^1=\overline{C_{0, \sigma}^{\infty}} \text { closure in the norm of } H^1.
\end{array}\right.
$$

Now, we can give our main results of this paper, which indicate the global existence and large-time behavior of strong solution to the system $\eqref{MHD}$-$\eqref{ydxw}$ in the exterior of a simply connected bounded smooth domain $D$ in $\mathbb{R}^3$.
\begin{theorem}
Let $\Omega$ be the exterior of a simply connected bounded domain $D$ in $\mathbb{R}^3$ and its boundary $\partial\Omega$ is smooth. For constants $\bar{\rho}>0, q \in(3, \infty)$, $p\in[1,12/7)$, assume that the initial data $\left(\rho_0, m_0, B_0\right)$ satisfy
\begin{equation}\label{cztj}
\begin{cases}
0 \leqslant \rho_0 \leqslant \bar{\rho}, \quad \rho_0 \in L^{3 / 2} \cap D^1, \quad \nabla \mu(\rho_0) \in L^q, \\
u_0 \in D_{0, \sigma}^1, \quad B_0 \in L^{p} \cap H_{0, \sigma}^1, \quad m_0=\rho_0 u_0 .
\end{cases}
\end{equation}
Then for
$$
\underline{\mu} \triangleq \min _{\rho \in[0, \bar{\rho}]} \mu(\rho), \quad \bar{\mu} \triangleq \max _{\rho \in[0, \bar{\rho}]} \mu(\rho), \quad M \triangleq\|\nabla \mu(\rho_0)\|_{L^q},
$$
there exists some small positive constant $\varepsilon_0$ depending only on $p, q, \bar{\rho}, \underline{\mu}, \bar{\mu}, \nu, \Omega, \|\rho_0\|_{L^{3 / 2}},\\
\|B_0\|_{L^{p}}$ and $M$ such that if
\begin{align}\label{du0b}
\|\nabla u_0\|_{L^2}+\|\nabla B_0\|_{L^2} \leqslant \varepsilon_0,
\end{align}
the system $\eqref{MHD}$-$\eqref{ydxw}$ admits a unique global strong solution $(\rho, u, B, P)$ satisfying that for any $0<\tau<T<\infty$, and any $p\in[2,s)$ with $s \triangleq min\{6,q\}$,
\begin{equation}\label{jdxw}
\begin{cases}
0 \leqslant \rho \in C([0, T] ; L^{3 / 2} \cap H^1)\cap L^\infty(\Omega\times(0,T)),\quad \nabla \mu(\rho) \in C([0, T] ; L^q), \\
\nabla u \in L^\infty(0,T;L^2)\cap L^\infty(\tau,T;W^{1,s})\cap C([\tau,T];H^1\cap W^{1,p}),\\
P \in L^\infty(\tau,T;W^{1,s})\cap C(\tau,T;H^1\cap W^{1,p}),\\
B \in L^{\infty}(\tau, T ; H^2) \cap L^2(\tau, T ; D^1\cap D^2) \cap C([\tau, T] ; H^2),\\
\rho^{1/2}u_t \in L^2(0,T;L^2)\cap L^\infty(\tau,T;L^2), \quad P_t \in L^2(\tau, T ; L^2\cap L^s),\\
B_t \in L^2(\tau,T;H^2)\cap L^\infty(\tau,T;H^1),\quad B_{tt} \in L^2(\tau,T;L^2),\\
\nabla B_t \in L^{\infty}(\tau, T ; L^2) \cap L^2(\tau, T ; H^1),\\
\nabla u_t \in L^{\infty}(\tau, T ; L^2) \cap L^2(\tau, T ; L^s),\quad (\rho u_t)_t\in L^2(\tau,T;L^2).
\end{cases}
\end{equation}
Moreover, it holds that
\begin{align}\label{dmidj}
\sup _{0 \leqslant t<\infty}\|\nabla \rho\|_{L^2} \leqslant 2\|\nabla \rho_0\|_{L^2}, \quad\sup _{0 \leqslant t<\infty}\|\nabla \mu(\rho)\|_{L^q} \leqslant 2\|\nabla \mu(\rho_0)\|_{L^q},
\end{align}
and that there exists some positive constant $C$ depending only on $\delta, p, q, \bar{\rho}, \underline{\mu}, \bar{\mu}, \nu, \Omega,$ $\|\rho_0\|_{L^{3 / 2} \cap D^1}, \|B_0\|_{L^{p}}$ and $M$ such that for all $t \geqslant 1$, for any $0<\delta<1/8$,
\begin{equation}\label{dutgui}
\begin{cases}
t^{\frac{6}{p}-\delta}\|\nabla u_t(\cdot, t)\|_{L^2}^2+t^{\frac{12-3p}{2p}-\delta}\|\nabla u(\cdot, t)\|_{H^1\cap W^{1,s}}^2+t^{\frac{12-3p}{2p}-\delta}\|P(\cdot, t)\|_{H^1\cap W^{1,s}}^2\leqslant C,\\
t^{\frac{6+3p}{2p}-\delta}\|\nabla B_t(\cdot, t)\|_{L^2}^2+t^{\frac{6+p}{2p}-\delta}\|B_t(\cdot, t)\|_{L^2}^2+t^{\frac{6-p}{2p}-\delta}\|\nabla B(\cdot, t)\|_{H^1}^2+t^{\frac{6-3p}{2p}}\|B(\cdot, t)\|_{L^2}^2 \leqslant C .
\end{cases}
\end{equation}
\end{theorem}

Moreover, if the viscosity $\mu(\rho)=\mu$ is a positive constant, we can also use our method to obtain the following global existence and large-time behavior of the strong solutions.
\begin{theorem}
For constants $\bar{\rho}>0$ and $\mu>0$, assume that $\mu(\rho) \equiv \mu$ in $\eqref{MHD}$ and the initial data $(\rho_0, u_0, B_0)$ satisfy $\eqref{cztj}$ except $B_0 \in L^{p}$. Then, there exists a unique global strong solution to the system $\eqref{MHD}$-$\eqref{ydxw}$ satisfying $\eqref{jdxw}$  with $s=6$ provided $\|\nabla u_0\|_{L^2}+\|\nabla B_0\|_{L^2} \leqslant \mu \varepsilon$. Moreover, it holds that
\begin{align}\label{cmgj}
\sup _{0 \leqslant t<\infty}\|\nabla \rho\|_{L^2} \leqslant 2\|\nabla \rho_0\|_{L^2},
\end{align}
and that there exists some positive constant C depending only on $\bar{\rho}, \mu, \nu,\|\rho_0\|_{L^{3 / 2}\cap D^1}$ and $\|B_0\|_{L^2}$  such that for all $t \geqslant 1$,
\begin{equation}\label{gdutgu}
\begin{cases}
t^{3}\|\nabla u_t(\cdot, t)\|_{L^2}^2+t^{\frac{3}{2}}\|\nabla u(\cdot, t)\|_{H^1\cap W^{1,6}}^2+t^{\frac{3}{2}}\|P(\cdot, t)\|_{H^1\cap W^{1,6}}^2\leqslant C,\\
t^3\|\nabla B_t(\cdot, t)\|_{L^2}^2+t^2\|B_t(\cdot, t)\|_{L^2}^2+t\|\nabla B(\cdot, t)\|_{H^1}^2 \leqslant C .
\end{cases}
\end{equation}
\end{theorem}
\begin{remark}
By comparison, we find that the decay rates for the solution $\|B\|_{L^2}$ and $\|\nabla B\|_{L^2}$ are the same as those for the linearized Navier-Stokes system (see \cite[Theorem 1.3]{TK1999}) and the heat equation (see \cite{Tre1975}). Therefore, these decay estimates are optimal.

Let's explain the range of $p$, for technical reasons, we need $\|B\|_{L^2}^{1/2}\|\nabla B\|_{L^2}^{3/2}\in L^1(0,T)$. 
Thanks to article \cite[Theorem 1.3]{TK1999}, roughly speaking, for the initial condition $B_0\in L^p$, 
the decay of $\|B\|_{L^2}^{1/2}\|\nabla B\|_{L^2}^{3/2}$ is $t^{-\frac{12-3p}{4p}}$, let $\frac{12-3p}{4p}>1$ we can get $p<\frac{12}{7}$.
\end{remark}
\begin{remark}
Compared with the results of Zhang \cite{Zha2020}, where he obtained the global strong solutions for the system $\eqref{MHD}$ on the Cauchy problem, we not only increase the algebraic decay rate of $\|\nabla u_t\|_{L^2}$, $\|\nabla u\|_{H^1\cap W^{1,s}}$, $\|P\|_{H^1\cap W^{1,s}}$, $\|\nabla B_t\|_{L^2}$ and $\|\nabla^2 B\|_{L^2}$,  but also get the new decay of $\|B\|_{L^2}$.
\end{remark}

\begin{remark}
The results obtained above hold for $\eqref{MHD}$ with the boundary conditions $B\cdot n=0, \curl B\times n=0$ on $\partial\Omega$ replaced by $B=0$ on $\partial\Omega$.
\end{remark}

We now make some comments on the analysis in this paper. The idea mainly comes from the article Cai-L\"u-Peng \cite{Cai2022} and Zhang \cite{Zha2020}. In this paper, 
the main difficulty is to obtain the suitable decay estimates of $B$ for exterior problems. Compared to the results by Zhang \cite{Zha2020}, instead of using the fundamental solutions of parabolic equations for Cauchy problem, 
on exterior domain we use method of energy integration $($see Lemma 2.7$)$ to obtain the decay of $B$.

The rest of this paper is organized as follows. Some facts and elementary inequalities are collected in Section 2. Section 3 is devoted to deriving a priori estimates of the system (1.1). Finally, we will give the proofs of Theorems 1.1 and 1.2 in Section 4.
\section{Auxiliary lemma}
In this section, we list some auxiliary lemmas which will be used later.

First of all, we start with the local existence of strong solutions which has been proved in \cite{Son2018}.
\begin{lemma}
Assume that $(\rho_0, u_0, B_0)$ satisfies $\eqref{cztj}$ except $B_0 \in L^{p}$. Then there exists a small time $T_0>0$ such that the problem $\eqref{MHD}$-$\eqref{ydxw}$ has a unique strong solution $(\rho, u, B, P)$ in $\Omega \times(0, T_0)$ satisfying $\eqref{jdxw}$.
\end{lemma}

Next, the following well-known Gagliardo-Nirenberg inequality will be used later frequently (see \cite{Ada2003,Lad1968}).
\begin{lemma}\textbf{(Gagliardo-Nirenberg)} 
Assume that $\Omega$ is the exterior of a simply connected domain $D$ in $\mathbb{R}^3$. For $p \in[2,6], q\in(1, \infty)$ and $r \in(3, \infty)$, there exists some generic constant $C>0$ which may depend on $p, q$ and $r$ such that for any $f \in H^1(\Omega)$ and $g \in L^q(\Omega) \cap D^{1, r}(\Omega)$, we have
\begin{gather}
\|f\|_{L^p(\Omega)} \leq C\|f\|_{L^2}^{\frac{6-p}{2 p}}\|\nabla f\|_{L^2}^{\frac{3 p-6}{2 p}},\label{G1}\\
\|g\|_{C(\bar{\Omega})} \leq C\|g\|_{L^q}^{q(r-3) /(3 r+q(r-3))}\|\nabla g\|_{L^r}^{3 r /(3 r+q(r-3))}.\label{G2}
\end{gather}
\end{lemma}
Generally, $\eqref{G1}$ and $\eqref{G2}$ are called Gagliardo-Nirenberg's inequalities.

The following two lemmas are given in \cite[Theorem 3.2]{Wahl1992} and \cite[Theorem 5.1]{Lou2016}.
\begin{lemma}
Assume that $\Omega$ is an exterior domain of some simply connected domain $\mathbb{D}$ in $\mathbb{R}^3$ with $C^{1,1}$ boundary. For $v \in D^{1, q}(\Omega)$ with $v \cdot n=0$ on $\partial \Omega$, it holds that
\begin{align}\label{dvgja}
\|\nabla v\|_{L^q(\Omega)} \leq C\left(\|\operatorname{div}v\|_{L^q(\Omega)}+\|\operatorname{curl} v\|_{L^q(\Omega)}\right) \text { for any } 1<q<3,
\end{align}
and
$$
\|\nabla v\|_{L^q(\Omega)} \leq C\left(\|\operatorname{div} v\|_{L^q(\Omega)}+\|\operatorname{curl} v\|_{L^q(\Omega)}+\|\nabla v\|_{L^2(\Omega)}\right) \text { for any } 3 \leq q<+\infty .
$$
\end{lemma}
\begin{lemma}
Assume that $\Omega$ is an exterior domain of some simply connected domain $\mathbb{D}$ in $\mathbb{R}^3$ with $C^{1,1}$ boundary. For any $v \in W^{1, q}(\Omega)(1<q<+\infty)$ with $v \times n=0$ on $\partial \Omega$, it holds that
$$
\|\nabla v\|_{L^q(\Omega)} \leq C\left(\|v\|_{L^q(\Omega)}+\|\operatorname{div} v\|_{L^q(\Omega)}+\|\operatorname{curl}v\|_{L^q(\Omega)}\right) .
$$
\end{lemma}

Moreover, we have the following conclusion (see \cite[Lemma 2.9]{Cai5586}).
\begin{lemma}
Assume that $\Omega$ is an exterior domain of some simply connected domain $\mathbb{D}$ in $\mathbb{R}^3$ with $C^{1,1}$ boundary. For $v \in D^{k+1, p}(\Omega) \cap D^{1,2}(\Omega)$ for some $k \geq 0, p \in[2,6]$ with $v \cdot n=0$ or $v \times n=0$ on $\partial \Omega$ and $v(x, t) \rightarrow 0$ as $|x| \rightarrow \infty$, there exists a positive constant $C=C(q, k, \mathbb{D})$ such that
\begin{align}
\|\nabla v\|_{W^{k, p}(\Omega)} \leq C\left(\|\operatorname{div}v\|_{W^{k, p}(\Omega)}+\|\operatorname{curl} v\|_{W^{k, p}(\Omega)}+\|\nabla v\|_{L^2(\Omega)}\right).
\end{align}
\end{lemma}
\begin{remark}
Noticing that $B \cdot n=0, \operatorname{curl} B \times n=0$ on $\partial \Omega$, by Lemma 2.3, for any integer $k \geq 1, p \in[2,6]$, we obtain
\begin{equation}
\begin{split}\label{dbkg}
\|\nabla B\|_{W^{k, p}} & \leq C\left(\|\operatorname{curl} B\|_{W^{k, p}}+\|\nabla B\|_{L^2}\right) \\
& \leq C\left(\left\|\operatorname{curl}^2 B\right\|_{W^{k-1, p}}+\|\operatorname{curl} B\|_{L^p}+\|\nabla B\|_{L^2}\right),
\end{split}
\end{equation}
where $\operatorname{curl}^2 B \triangleq$ $\operatorname{curl}\operatorname{curl}$ $B$ and we have used the fact $\operatorname{div}\operatorname{curl}$ $B=0$.
\end{remark}

The following regularity results on the Stokes equations will be useful for our derivation of higher order a priori estimates (see \cite[Lemma 2.11]{Cai2022}).
\begin{lemma}\label{le23}
For positive constants $\underline{\mu}, \bar{\mu}$, and $q \in(3, \infty)$, in addition to $\eqref{vis}$, assume that $\mu(\rho)$ satisfies
\begin{align}
\nabla \mu(\rho) \in L^q, \quad 0<\underline{\mu} \leq \mu(\rho) \leq \bar{\mu}<\infty .
\end{align}
For the problem with the boundary condition $\eqref{ubj}$
\begin{equation}\label{sto}
\begin{cases}
-\operatorname{div}\left(2 \mu(\rho) D(u)\right)+\nabla P=F, & x \in \Omega, \\ \operatorname{div} u=0, & x \in \Omega, \\ u(x) \rightarrow 0, & |x| \rightarrow \infty,
\end{cases}
\end{equation}
we have the following conclusions:

(1) If $F=f \in L^{6 / 5} \cap L^r$ with $r \in[2 q /(q+2), q]$, then there exists some positive constant $C$ depending only on $\underline{\mu}, \bar{\mu}, r$, and $q$ such that the unique weak solution $(u, P) \in$ $D_{0, \mathrm{div}}^1 \times L^2$ satisfies
\begin{gather}
\|\nabla u\|_{L^2}+\|P\|_{L^2} \leq C\|f\|_{L^{6 / 5}}, \label{st1}\\
\|\nabla^2 u\|_{L^r}+\|\nabla P\|_{L^r} \leq C\|f\|_{L^r}+C\left(\|\nabla \mu(\rho)\|_{L^q}^{\frac{q(5 r-6)}{2r(q-3)}}+1\right)\|f\|_{L^{6 / 5}} .
\end{gather}

(2) If $F=\operatorname{div} g$ with $g \in L^2 \cap L^{\tilde{r}}$ for some $\tilde{r} \in(6 q /(q+6), q]$, then there exists a positive constant $C$ depending only on $\underline{\mu}, \bar{\mu}, q$, and $\tilde{r}$ such that the unique weak solution $(u, P) \in D_{0, \mathrm{div}}^1 \times L^2$ to $\eqref{sto}$ satisfies
\begin{align}
\|\nabla u\|_{L^2 \cap L^{\tilde{r}}}+\|P\|_{L^2 \cap L^{\tilde{r}}} \leq C\|g\|_{L^2 \cap L^{\tilde{r}}}+C\|\nabla \mu(\rho)\|_{L^q}^ \frac{3 q(\tilde{r}-2)}{2\tilde{r}(q-3)}\|g\|_{L^2}.
\end{align}
\end{lemma}
%

Finally, we state the following elementary inequality(see \cite[Lemma 2.2]{Fan2019}), which will be used in the proof of decay estimates  for $\|B\|_{L^2}$. Furthermore, in order to get the boundary estimates, one can extend the unit normal vector $n$ to $R^3$ such that $n\in C_0^\infty(R^3)$ and $n\equiv 0$ outside $B_{2R_0}$ $\big($let $R_0$ sufficiently large such that $R^3\backslash \Omega\subset B_{R_0}$$\big)$.


\begin{lemma}
Let $\Omega$ be a regular exteior domain in $\mathbb{R}^3, b: \Omega \rightarrow \mathbb{R}^3$ be a sufficiently smooth vector field, $b(x) \rightarrow 0$ as $|x| \rightarrow \infty$, and let $1<p<\infty$. Then, there exists a positive constant C depending only on p and $\Omega$ such that for $\sigma=min\{p-1,1\}$,
$$
\begin{aligned}
-\int_{\Omega} \Delta b \cdot b|b|^{p-2} \mathrm{~d} x
\geqslant\frac{2\sigma}{p^2}\int_{\Omega}|\nabla |b|^\frac{p}{2}|^2 \mathrm{~d} x-C\int_{\Omega\cap B_{2R_0}}|b|^p\mathrm{~d} x.
\end{aligned}
$$
In particular, when $1<p\leqslant6$, we have
\begin{align}\label{sx1}
-\int_{\Omega} \Delta b \cdot b|b|^{p-2} \mathrm{~d} x \geqslant\frac{2\sigma}{p^2}\int_{\Omega}|\nabla |b|^\frac{p}{2}|^2 \mathrm{~d} x-C\|\nabla b\|_{L^2}^p.
\end{align}
\end{lemma}
\begin{proof}
Applying the fact that $|\nabla |b|^\frac{p}{2}|^2\leqslant (\frac{p}{2})^2|b|^{p-2}|\nabla b|^2$ gives
\begin{equation}
\begin{split}\nonumber
& -\int_{\Omega} \Delta b \cdot b|b|^{p-2} \mathrm{~d} x \\
= & \int_{\Omega}|b|^{p-2}|\nabla b|^2 \mathrm{~d} x+4 \frac{p-2}{p^2} \int_{\Omega}|\nabla| b|^{\frac{p}{2}}|^2 \mathrm{~d} x-\int_{\partial \Omega}|b|^{p-2}b \cdot \nabla b \cdot n \mathrm{~d} S\\
\geqslant&\sigma\int_{\Omega}|b|^{p-2}|\nabla b|^2 \mathrm{~d} x+\int_{\partial \Omega}|b|^{p-2}b \cdot\nabla n \cdot b \mathrm{~d} S\\
\geqslant&\frac{4\sigma}{p^2}\int_{\Omega}|\nabla |b|^\frac{p}{2}|^2 \mathrm{~d} x-\int_{\partial\Omega}\left(|b|^\frac{p}{2}\cdot|\nabla n|^\frac{1}{2}\right)^2\mathrm{~d} S\\
\geqslant&\frac{4\sigma}{p^2}\int_{\Omega}|\nabla |b|^\frac{p}{2}|^2 \mathrm{~d} x-\varepsilon\int_{\Omega}\left|\nabla \left(|b|^\frac{p}{2}\cdot|\nabla n|^\frac{1}{2}\right)\right|^2 \mathrm{~d} x-C_\varepsilon\int_{\Omega}\left(|b|^\frac{p}{2}\cdot|\nabla n|^\frac{1}{2}\right)^2\mathrm{~d} x\\
\geqslant&\frac{2\sigma}{p^2}\int_{\Omega}|\nabla |b|^\frac{p}{2}|^2 \mathrm{~d} x-C\int_{\Omega\cap B_{2R_0}}|b|^p\mathrm{~d} x.
\end{split}
\end{equation}
\end{proof}

\section{A priori estimates}
In this section, $\Omega$ is always the exterior of a simply connected bounded smooth domain $D$ in $\mathbb{R}^3$. we will establish some necessary a priori bounds of local strong solutions $(\rho, u, B, P)$ to the system $\eqref{MHD}$-$\eqref{ydxw}$ whose existence is guaranteed by Lemma 2.1. Thus, let $T>0$ be a fixed time and $(\rho, u, B, P)$ be the smooth solution to $\eqref{MHD}$-$\eqref{ydxw}$ on $\Omega \times(0, T]$ with smooth initial data $(\rho_0, u_0, B_0)$ satisfying \eqref{cztj}. In particular, we sometimes use $C(\delta)$ to mean that $C$ depends on $\delta$.

We have the following key a priori estimates on $(\rho, u, B, P)$:
\begin{proposition}\label{p0}
There exists some positive constant $\varepsilon_0$ depending only on $p, q, \bar{\rho}, \underline{\mu}, \bar{\mu},\\ \nu, \Omega, \|\rho_0\|_{L^{3/2}}$, $\|B_0\|_{L^{p}}$ and $M$ such that if $(\rho, u, B, P)$ is a smooth solution of $\eqref{MHD}$-$\eqref{ydxw}$ on $\Omega \times(0, T]$ satisfying
\begin{equation}
\begin{split}\label{p1}
&\int_0^T\|\nabla u\|_\infty \mathrm{~d} t \leqslant \frac{\ln 4}{q}, \quad \int_0^T\|\nabla B\|_{L^2}^p\mathrm{~d} t\leqslant 2\text {, }\\
&\int_0^T\|\nabla u\|_{L^2}^4 \mathrm{~d} t \leqslant 2(\|\nabla u_0\|_{L^2}^2+\|\nabla B_0\|_{L^2}^2),
\end{split}
\end{equation}
the following estimates hold:
\begin{equation}
\begin{split}\label{p2}
&\int_0^T\|\nabla u\|_\infty \mathrm{~d} t \leqslant \frac{\ln 2}{q}, \quad \int_0^T\|\nabla B\|_{L^2}^p\mathrm{~d} t\leqslant 1\text {, }\\
&\int_0^T\|\nabla u\|_{L^2}^4 \mathrm{~d} t \leqslant \|\nabla u_0\|_{L^2}^2+\|\nabla B_0\|_{L^2}^2,
\end{split}
\end{equation}
provided that $\|\nabla u_0\|_{L^2}+\|\nabla B_0\|_{L^2}\le \varepsilon_0$.
\end{proposition}

Before proving Proposition 3.1, we establish some necessary a priori estimates, see Lemmas 3.2-3.7.

We begin with the following estimates:
\begin{lemma}
Let $(\rho, u, B, P)$ be a smooth solution to $\eqref{MHD}$-$\eqref{ydxw}$ satisfying $\eqref{p1}$. Then there exists a generic positive constant $C$ depending only on $p, \underline\mu, \nu, \|\rho_0\|_{L^{3/2}}$ and $\|B_0\|_{L^{p}}$ such that
\begin{align}
\sup _{t \in [0, T]}\left(\|\rho^{1 / 2} u\|_{L^2}^2+\|B\|_{L^2}^2\right)+\int_0^T\left(\|\nabla u\|_{L^2}^2+\|\nabla B\|_{L^2}^2\right) \mathrm{~d} t \leqslant C,\label{xy}
\end{align}
provided $\|\nabla u_0\|_{L^2}+||\nabla B_0\|_{L^2}\leqslant 1$.
\end{lemma}

\begin{proof}
Multiplying $\eqref{MHD}_2$ and $\eqref{MHD}_3$ by $u$ and $B$, respectively, and then integrating the resulting equality by parts over $\Omega$ leads to
$$
\frac{1}{2} \frac{d}{d t}\left(\|\rho^{1 / 2} u\|_{L^2}^2+\|B\|_{L^2}^2\right)+\int\left(2 \mu(\rho)|d|^2+\nu|\operatorname{curl} B|^2\right) \mathrm{~d} x=0,
$$
by Lemma 2.2, one easily deduces that
\begin{equation}
\begin{cases}\label{ub0}
&\|\rho_0^{1/2}u_0\|_{L^2}^2\leqslant\|\rho_0\|_{L^{3/2}}\|\nabla u_0\|_{L^2}^2,\\
&\|B_0\|_{L^2}^2\leqslant\|B_0\|_{L^{p}}^\frac{4p}{6-p}\|\nabla B_0\|_{L^2}^\frac{12-6p}{6-p},
\end{cases}
\end{equation}
thus it follows from $\eqref{dvgja}$ and $\eqref{ub0}$ that $\eqref{xy}$ holds provided $\|\nabla u_0\|_{L^2}+||\nabla B_0\|_{L^2}\leqslant 1$. The proof of Lemma 3.2 is finished.
\end{proof}

\begin{lemma}
Let $(\rho, u, B, P)$ be a smooth solution to $\eqref{MHD}$-$\eqref{ydxw}$ satisfying $\eqref{p1}$. Then there exists a generic positive constant $C$ depending only on $p, \underline{\mu}, \nu, \|\rho_0\|_{L^{3/2}}$ and $\|B_0\|_{L^{p}}$ such that
\begin{align}
&\sup _{t \in \mid 0, T]}\|\nabla B\|_{L^2}^2+\int_0^T\left(\|B_t\|_{L^2}^2+\|\Delta B\|_{L^2}^2\right)\mathrm{~d} t \leqslant C\|\nabla B_0\|_{L^2}^2 ,\label{dbnlg}\\
&\sup _{t \in[0, T]}t\|\nabla B\|_{L^2}^2+\int_0^T t\left(\|B_t\|_{L^2}^2+\|\Delta B\|_{L^2}^2\right)\mathrm{~d} t \leqslant C ,\label{dbnlt}
\end{align}
provided $\|\nabla u_0\|_{L^2}+||\nabla B_0\|_{L^2}\leqslant 1$.
\end{lemma}
\begin{proof}
Thanks to $\eqref{MHD}_3$, Lemma 2.2 and $\eqref{dbkg}$, using integration by parts, we derive that
$$
\begin{aligned}
&\nu \frac{\mathrm{d}}{\mathrm{d} t}\|\operatorname{curl} B\|_{L^2}^2+\|B_t\|_{L^2}^2+\nu^2\|\Delta B\|_{L^2}^2 \\
& =\int|B_t-\nu \Delta B|^2 \mathrm{~d} x=\int|B \cdot \nabla u-u \cdot \nabla B|^2 \mathrm{~d} x \\
& \leqslant \frac{\nu^2}{2}\|\Delta B\|_{L^2}^2+C\left(\|\nabla u\|_{L^2}^4+\|\nabla u\|_{L^2}^2\right)\|\nabla B\|_{L^2}^2,
\end{aligned}
$$
which, together with $\eqref{dvgja}$, gives
\begin{align}\label{dbg}
\nu\frac{\mathrm{d}}{\mathrm{d} t}\|\operatorname{curl} B\|_{L^2}^2+\|B_t\|_{L^2}^2+\frac{\nu^2}{2}\|\Delta B\|_{L^2}^2\leqslant C\left(\|\nabla u\|_{L^2}^4+\|\nabla u\|_{L^2}^2\right)\|\operatorname{curl} B\|_{L^2}^2,
\end{align}
this combined with Gr\"onwall's inequality, $\eqref{p1}$ and $\eqref{xy}$ yields
\begin{align}
\sup _{t \in \mid 0, T]}\|\nabla B\|_{L^2}^2+\int_0^T\left(\|B_t\|_{L^2}^2+\|\Delta B\|_{L^2}^2\right) \mathrm{~d} t \leqslant C\|\nabla B_0\|_{L^2}^2 .\label{dbnlj}
\end{align}

Furthermore, multiplying $\eqref{dbg}$ by $t$ leads to
$$
\begin{aligned}
&\nu\frac{\mathrm{d}}{\mathrm{d} t}\left(t\|\operatorname{curl} B\|_{L^2}^2\right)+t\Big(\|B_t\|_{L^2}^2+\frac{\nu^2}{2}\|\Delta B\|_{L^2}^2\Big)\\
&\leqslant C\left(\|\nabla u\|_{L^2}^4+\|\nabla u\|_{L^2}^2\right)\left(t\|\operatorname{curl} B\|_{L^2}^2\right)+\nu\|\operatorname{curl} B\|_{L^2}^2,
\end{aligned}
$$
combining this with Gr\"onwall's inequality, $\eqref{dvgja}$, $\eqref{p1}$ and $\eqref{xy}$ shows that
\begin{align}
\sup _{t \in[0, T]}t\|\nabla B\|_{L^2}^2+\int_0^T t\left(\|B_t\|_{L^2}^2+\|\Delta B\|_{L^2}^2\right) \mathrm{~d} t \leqslant C .
\end{align}
The proof of Lemma 3.3 is finished.
\end{proof}

The following decay estimate of $B$ is crucial for the subsequent analysis:
\begin{lemma}
Let $(\rho, u, B, P)$ be a smooth solution to $\eqref{MHD}$-$\eqref{ydxw}$ satisfying $\eqref{p1}$. Then there exists a generic positive constant $C$ depending only on $p, \underline{\mu}, \nu, \Omega, \|\rho_0\|_{L^{3/2}}$ and $\|B_0\|_{L^{p}}$ such that for any $0<\delta<1/8$, one obtain
\begin{align}
&\sup_{t\in[0,T]}t^{\frac{6-p}{2p}-\delta}\|\nabla B\|_{L^2}^2+\int_0^T t^{\frac{6-p}{2p}-\delta}\left(\|B_t\|_{L^2}^2+\|\Delta B\|_{L^2}^2\right)\mathrm{~d} t\leqslant C(\delta),\label{dbsj}\\
&\sup_{t\in [0,T]}\left(t^{\frac{12-5p}{2p}-\delta}\|\rho^{1/2}u\|_{L^2}^2+t^\frac{6-3p}{2p}\|B\|_{L^2}^2\right)+\int_0^T t^{\frac{12-5p}{2p}-\delta}\|\nabla u\|_{L^2}^2\mathrm{~d} t\leqslant C(\delta),\label{usj}
\end{align}
provided $\|\nabla u_0\|_{L^2}+||\nabla B_0\|_{L^2}\leqslant 1$.
\end{lemma}
\begin{proof}
First, standard arguments (see \cite{Lio1996}) imply that
\begin{align}\label{midu}
0\leqslant\rho\leqslant\overline{\rho},\quad\|\rho\|_{L^{3/2}}=\|\rho_0\|_{L^{3/2}}.
\end{align}

Next, multiplying $\eqref{MHD}_3$ by $B {|B|}^{p-2}$ $($$1<p<\frac{12}{7}$$)$(by interpolation inequality, we only consider the case $p>1$), we obtain after using integration by parts and $\eqref{sx1}$ that
$$
\begin{aligned}
&\frac{1}{p} \frac{d}{d t}\int |B|^{p}\mathrm{~d} x+\frac{2\nu(p-1)}{p^2} \int |\nabla |B|^\frac{p}{2}|^2 \mathrm{~d} x\\
&\leqslant \int(B\cdot\nabla u-u\cdot\nabla B)\cdot B {|B|}^{p-2}\mathrm{~d} x+C\|\nabla B\|_{L^2}^p\\
&\leqslant C\|\nabla u\|_{L^\infty}\int |B|^{p}\mathrm{~d} x+C\|\nabla B\|_{L^2}^p,
\end{aligned}
$$
which, together with Gr\"onwall's inequality and $\eqref{p1}$, yields
\begin{align}
\sup _{t \in[0, T]}\|B\|_{L^{p}}^{p}+\int_0^T\left\|\nabla |B|^\frac{p}{2}\right\|_{L^2}^2\mathrm{~d} t \leqslant C .\label{cs}
\end{align}
Observe that
$$
\|B\|_{L^2}=\||B|^\frac{p}{2}\|_{L^\frac{4}{p}}^\frac{2}{p}\leqslant\|\nabla|B|^\frac{p}{2}\|_{L^2}^\frac{6-3p}{2p}\||B|^
\frac{p}{2}\|_{L^2}^\frac{3p-2}{2p},
$$
which also can deduce
\begin{align}\label{4bgj}
\int_0^T\|B\|_{L^2}^\frac{4p}{6-3p}\mathrm{~d} t\leqslant\int_0^T\left\|\nabla |B|^\frac{p}{2}\right\|_{L^2}^2\|B\|_{L^{p}}^\frac{p(3p-2)}{6-3p}\mathrm{~d} t\leqslant C.
\end{align}

Testing $\eqref{MHD}_3$ by $B$ and using $\eqref{G1}$, we see that
\begin{align}
\frac{d}{d t}\|B\|_{L^2}^2+\nu\|\nabla B\|_{L^2}^2 \leqslant C\|\nabla u\|_{L^2}^4\|B\|_{L^2}^2,\label{bnlg}
\end{align}
multiplying $\eqref{bnlg}$ by $t\|B\|_{L^2}^\frac{10p-12}{6-3p}$, one easily obtain
\begin{align}\label{ch2j}
\frac{d}{dt}\big(t\|B\|_{L^2}^\frac{4p}{6-3p}\big)\leqslant C\|\nabla u\|_{L^2}^4\big(t\|B\|_{L^2}^\frac{4p}{6-3p}\big)+\|B\|_{L^2}^\frac{4p}{6-3p},
\end{align}
which combining with $\eqref{p1}$ and $\eqref{4bgj}$, one obtain
\begin{align}\label{bdgg}
\sup_{t\in[0,T]}t\|B\|_{L^2}^\frac{4p}{6-3p}\leqslant C.
\end{align}

Multiplying $\eqref{dbg}$ and $\eqref{bnlg}$ by $t^{\frac{6-p}{2p}-\delta}$ and $t^{\frac{6-3p}{2p}-\delta}$, respectively, combining those results inequality one obtain
\begin{align}\label{pebg}
& \frac{\mathrm{d}}{\mathrm{d} t}\left(t^{\frac{6-p}{2p}-\delta}\|\operatorname{curl} B\|_{L^2}^2+t^{\frac{6-3p}{2p}-\delta}\|B\|_{L^2}^2\right)+Ct^{\frac{6-p}{2p}-\delta}\left(\|B_t\|_{L^2}^2+\|\Delta B\|_{L^2}^2+t^{-1}\|\nabla B\|_{L^2}^2\right)\nonumber\\
&\leqslant C\left(\|\nabla u\|_{L^2}^4+\|\nabla u\|_{L^2}^2\right)\left(t^{\frac{6-p}{2p}-\delta}\|\operatorname{curl} B\|_{L^2}^2+t^{\frac{6-3p}{2p}-\delta}\|B\|_{L^2}^2\right)+Ct^{\frac{6-5p}{2p}-\delta}\|B\|_{L^2}^2,
\end{align}
observe for any $0<\delta<\frac{1}{8}$, $\frac{6-5p}{2p}-\delta>-1$, hence we have
$$
\begin{aligned}
&\int_0^Tt^{\frac{6-5p}{2p}-\delta}\|B\|_{L^2}^2\mathrm{~d} t\\
&\leqslant \sup_{t\in[0,1]}\|B\|_{L^2}^2\int_0^1t^{\frac{6-5p}{2p}-\delta} \mathrm{~d} t +\sup_{t\in[1,T]}\left(t^\frac{6-3p}{2p}\|B\|_{L^2}^2\right)\int_1^Tt^{-1-\delta}\mathrm{~d} t\\
&\leqslant C(\delta).
\end{aligned}
$$
Thus, integrating $\eqref{pebg}$ over $(0,T)$ and using estimate $\eqref{dvgja}$ and $\eqref{p1}$ and $\eqref{xy}$, we get
\begin{align}\label{dbs1}
\sup_{t\in[0,T]}t^{\frac{6-p}{2p}-\delta}\|\nabla B\|_{L^2}^2+\int_0^T t^{\frac{6-p}{2p}-\delta}\left(\|B_t\|_{L^2}^2+\|\Delta B\|_{L^2}^2+t^{-1}\|\nabla B\|_{L^2}^2\right)\mathrm{~d} t\leqslant C(\delta).
\end{align}

Finally, multiplying $\eqref{MHD}_2$ by $t^{\frac{12-5p}{2p}-\delta}u$, and integrating the resulting equations by parts over $\Omega$ leads to
$$
\frac{d}{dt}\left(t^{\frac{12-5p}{2p}-\delta}\|\rho^{1/2} u\|_{L^2}^2\right)+2\underline\mu t^{\frac{12-5p}{2p}-\delta}\|\nabla u\|_{L^2}^2\leqslant Ct^{\frac{12-5p}{2p}-\delta}\|B\|_{L^2}\|\nabla B\|_{L^2}^3+Ct^{\frac{12-7p}{2p}-\delta}\|\nabla u\|_{L^2}^2,
$$
%
it follows from $\eqref{bdgg}$ and $\eqref{dbs1}$, we have
$$
\begin{aligned}
&\int_0^Tt^{\frac{12-5p}{2p}-\delta}\|B\|_{L^2}\|\nabla B\|_{L^2}^3\mathrm{~d} t\\
&\leqslant \sup_{t\in[0,T]}t^\frac{6-3p}{4p}\|B\|_{L^2}\sup_{t\in[0,T]}t^{\frac{6-p}{4p}-\frac{\delta}{2}}\|\nabla B\|_{L^2}\int_0^Tt^{\frac{6-3p}{2p}-\frac{\delta}{2}}\|\nabla B\|_{L^2}^2\mathrm{~d} t\leqslant C(\delta),
\end{aligned}
$$
for $t^{\frac{12-7p}{2p}-\delta}\|\nabla u\|_{L^2}^2$, when $0<t\leqslant1$,
$$
\int_0^1 t^{\frac{12-7p}{2p}-\delta}\|\nabla u\|_{L^2}^2 \mathrm{~d} t\leqslant \left(\int_0^1 t^{\frac{12-7p}{p}-2\delta}\mathrm{~d} t\right)^{1/2}\left(\int_0^1\|\nabla u\|_{L^2}^4\mathrm{~d} t\right)^{1/2}\leqslant C,
$$
when $1<t\leqslant T$,
$$
\begin{aligned}
\int_1^T t^{\frac{12-7p}{2p}-\delta}\|\nabla u\|_{L^2}^2 \mathrm{~d} t
&\leqslant \underline\mu\int_1^T t^{\frac{12-5p}{2p}-\delta}\|\nabla u\|_{L^2}^2 \mathrm{~d} t+C\int_1^T t^{\frac{12-13p}{2p}-\delta}\|\nabla u\|_{L^2}^2 \mathrm{~d} t\\
&\leqslant \underline\mu\int_1^T t^{\frac{12-5p}{2p}-\delta}\|\nabla u\|_{L^2}^2 \mathrm{~d} t+C,
\end{aligned}
$$
hence we can deduce $\eqref{usj}$. The proof of Lemma 3.4 is finished.
\end{proof}

By Lemma 3.2-3.4, we can now derive the uniform bound for the gradient of velocity on $L^{\infty}\left(0, T ; L^2\right)$:
\begin{lemma}\label{lemma3.5}
Let $(\rho, u, B, P)$ be a smooth solution to $\eqref{MHD}$-$\eqref{ydxw}$ satisfying $\eqref{p1}$. Then there exists a generic positive constant $C$ depending only on $p, q, \bar{\rho}, \underline{\mu}, \bar{\mu}, \nu, \Omega, \|\rho_0\|_{L^{3/2}}$, $\|B_0\|_{L^{p}}$ and $M$ such that for any $0<\delta<1/8$, one obtain
\begin{align}
&\sup _{t \in[0,T]}\|\nabla u\|_{L^2}^2+\int_0^T \|\rho^{1 / 2} u_t\|_{L^2}^2 \mathrm{~d} t \leqslant C\left(\|\nabla u_0\|_{L^2}^2+\|\nabla B_0\|_{L^2}^2\right) ,\label{dux}\\
&\sup _{t \in[0,T]}t^{\frac{6-2p}{p}-\delta}\|\nabla u\|_{L^2}^2+\int_0^T t^{\frac{6-2p}{p}-\delta}\|\rho^{1 / 2} u_t\|_{L^2}^2 \mathrm{~d} t\leqslant C(\delta) ,\label{tdugj}
\end{align}
provided $\|\nabla u_0\|_{L^2}+||\nabla B_0\|_{L^2}\leqslant 1$.
\end{lemma}
\begin{proof}
First, it follows from $\eqref{MHD}_1$ that
$$
[\mu(\rho)]_t+u \cdot \nabla \mu(\rho)=0,
$$
by standard calculations, we have
\begin{align}\label{nqgj}
\frac{\mathrm{d}}{\mathrm{d} t}\|\nabla \mu(\rho)\|_{L^{q}} \leqslant q\|\nabla u\|_{L^{\infty}}\|\nabla \mu(\rho)\|_{L^q},
\end{align}
which together with Gr\"onwall's inequality and $\eqref{p1}$ yields
\begin{align}\label{mmdq}
\sup _{t \in[0, T]}\|\nabla \mu(\rho)\|_{L^q} \leqslant\|\nabla \mu(\rho_0)\|_{L^q} \exp \{q \int_0^T\|\nabla u\|_{L^{\infty}} \mathrm{d} t\}\leqslant 4\|\nabla \mu(\rho_0)\|_{L^q}.
\end{align}

Next, it follows from $\eqref{MHD}_2$, Lemma 2.2, Lemma 2.6, $\eqref{midu}$ and $\eqref{mmdq}$ that
$$
\begin{aligned}
&\|\nabla u\|_{H^1}+\|P\|_{H^1}\\
&\leqslant C\|\rho u_t+\rho u\cdot\nabla u-B\cdot\nabla B\|_{L^{6/5}\cap L^2}\\
&\leqslant C\left(\|\rho\|_{L^{3/2}}^{1/2}+\overline\rho^{1/2}\right)\left(\|\rho^{1/2}u_t\|_{L^2}+\overline\rho^{1/2}\|u\cdot\nabla u\|_{L^2}\right)+C\|B\|_{L^6}\|\nabla B\|_{L^3}+C\|B\|_{L^3}\|\nabla B\|_{L^2}\\
&\leqslant \frac{1}{2}\|\nabla^2 u\|_{L^2}+C\|\rho^{1/2}u_t\|_{L^2}+C\|\nabla u\|_{L^2}^3+C\|\nabla B\|_{L^2}^2+C\left(\|B\|_{L^2}^{1/2}+\|\Delta B\|_{L^2}^{1/2}\right)\|\nabla B\|_{L^2}^{3/2},
\end{aligned}
$$
which directly yields that
\begin{align}
&\|\nabla u\|_{H^1}+\|P\|_{H^1}\nonumber\\
&\leqslant C\|\rho^{1/2}u_t\|_{L^2}+C\|\nabla u\|_{L^2}^3+C\|\nabla B\|_{L^2}^2+C\left(\|B\|_{L^2}^{1/2}+\|\Delta B\|_{L^2}^{1/2}\right)\|\nabla B\|_{L^2}^{3/2}.\label{tht}
\end{align}

Multiplying $\eqref{MHD}_2$ by $u_t$ and integrating the resulting equality by parts leads to
\begin{equation}\label{duh2}
\begin{split}
&\frac{1}{4} \frac{\mathrm{d}}{\mathrm{d} t} \int \mu(\rho)|\nabla u+(\nabla u)^T|^2 \mathrm{~d} x+\int \rho|u_t|^2 \mathrm{~d} x \\
&=\frac{\mathrm{d}}{\mathrm{d} t} \int B \cdot \nabla B \cdot u \mathrm{~d} x + \frac{1}{4} \int \mu(\rho) u \cdot \nabla|\nabla u+(\nabla u)^T|^2 \mathrm{~d}x\\
&\quad -\int \rho u \cdot \nabla u \cdot u_t \mathrm{~d} x- \int \left(B_t \cdot \nabla B \cdot u + B \cdot \nabla B_t \cdot u\right) \mathrm{~d} x \\
& \triangleq \frac{\mathrm{d}}{\mathrm{d} t} J_0+\sum_{i=1}^3 J_i .
\end{split}
\end{equation}


Now, we will use Lemma 2.2, $\eqref{xy}$, $\eqref{dbnlg}$, $\eqref{midu}$ and $\eqref{tht}$ to estimate each term on the right hand of $\eqref{duh2}$ as follows:
$$
\begin{aligned}
|J_1|&=\left|\frac{1}{4}\int \mu(\rho)u\cdot\nabla |\nabla u+(\nabla u)^T|^2 \mathrm{~d} x\right|\\
&\leqslant C\|u\|_{L^6}\|\nabla u\|_{L^3}\|\nabla^2 u\|_{L^2}\\
&\leqslant C\|\nabla u\|_{L^2}^{3/2}\|\nabla^2 u\|_{L^2}^{3/2}\\
&\leqslant\frac{1}{4}\|\rho^{1/2}u_t\|_{L^2}^2+C\|\nabla u\|_{L^2}^6+C\|\nabla u\|_{L^2}^{3/2}\|\nabla B\|_{L^2}^3+C\|\nabla B\|_{L^2}\|\Delta B\|_{L^2}\|\nabla u\|_{L^2}^2\\
&\quad+C\|\nabla B\|_{L^2}^6+C\|\nabla u\|_{L^2}^{3/2}\|B\|_{L^2}^{3/4}\|\nabla B\|_{L^2}^{9/4},
\end{aligned}
$$
$$
\begin{aligned}
|J_2|&=\left|-\int \rho u \cdot \nabla u \cdot u_t \mathrm{~d} x\right|\\
&\leqslant C\|\rho^{1/2}u_t\|_{L^2}\|u\|_{L^6}\|\nabla u\|_{L^3}\\
&\leqslant\frac{1}{4}\|\rho^{1/2}u_t\|_{L^2}^2+C\|\nabla u\|_{L^2}^6+C\|\nabla u\|_{L^2}^3\|\nabla B\|_{L^2}^2+C\|\nabla u\|_{L^2}^3\|\nabla B\|_{L^2}^{3/2}\\
&\quad+C\|\nabla B\|_{L^2}^{3/2}\|\Delta B\|_{L^2}^{1/2}\|\nabla u\|_{L^2}^3\\
&\leqslant\frac{1}{4}\|\rho^{1/2}u_t\|_{L^2}^2+C\|\nabla u\|_{L^2}^6+C\|\nabla u\|_{L^2}^{3/2}\|\nabla B\|_{L^2}^2+C\|\nabla B\|_{L^2}\|\Delta B\|_{L^2}\|\nabla u\|_{L^2}^2\\
&\quad+C\|\nabla B\|_{L^2}^6,
\end{aligned}
$$
and
$$
\begin{aligned}
|J_3|&=\left|\- \int \left(B_t \cdot \nabla B \cdot u + B \cdot \nabla B_t \cdot u\right) \mathrm{~d} x \right|=\left|\- \int \left(B_t \cdot \nabla B \cdot u - B \cdot \nabla u \cdot B_t\right) \mathrm{~d} x \right|\\
&\leqslant C\|B_t\|_{L^2}\|\nabla B\|_{L^3}\|u\|_{L^6}+C\|B\|_{L^\infty}\|\nabla u\|_{L^2}\|B_t\|_{L^2}\\
&\leqslant C\|B_t\|_{L^2}\|\nabla B\|_{L^2}^{1/2}\|\nabla^2 B\|_{L^2}^{1/2}\|\nabla u\|_{L^2}\\
&\leqslant C\|B_t\|_{L^2}\|\nabla B\|_{L^2}^{1/2}\|\Delta B\|_{L^2}^{1/2}\|\nabla u\|_{L^2}+C\|B_t\|_{L^2}\|\nabla B\|_{L^2}\|\nabla u\|_{L^2}.\\
\end{aligned}
$$

Substituting $J_1$, $J_2$ and $J_3$ into $\eqref{duh2}$ gives
\begin{align}
&\frac{1}{4}\frac{\mathrm{d}}{\mathrm{d} t} \int\mu(\rho)\left|\nabla u+(\nabla u)^T\right|^2 \mathrm{d} x+\|\rho^{1 / 2} u_t\|_{L^2}^2 \nonumber\\
&\leqslant \frac{\mathrm{d}}{\mathrm{d} t} J_0+C\left(\|\nabla u\|_{L^2}^4+\|\nabla B\|_{L^2}\|\Delta B\|_{L^2}\right)\|\nabla u\|_{L^2}^2+C\|\nabla u\|_{L^2}^{3/2}\|\nabla B\|_{L^2}^2\nonumber\\
&\quad+C\|\nabla B\|_{L^2}^6+C\|B_t\|_{L^2}\|\nabla B\|_{L^2}^{1/2}\|\Delta B\|_{L^2}^{1/2}\|\nabla u\|_{L^2}+C\|B_t\|_{L^2}\|\nabla B\|_{L^2}\|\nabla u\|_{L^2}\nonumber\\
&\triangleq \frac{\mathrm{d}}{\mathrm{d} t} J_0+C\left(\|\nabla u\|_{L^2}^4+\|\nabla B\|_{L^2}\|\Delta B\|_{L^2}\right)\|\nabla u\|_{L^2}^2+C\sum_{i=1}^4R_i(t).\label{duhb}
\end{align}
It follows from $\eqref{p1}$, $\eqref{xy}$ and $\eqref{dbnlg}$ gives
\begin{equation}
\begin{split}
\int_0^T \left(\|\nabla u\|_{L^2}^4+\|\nabla B\|_{L^2}\|\Delta B\|_{L^2}\right)\mathrm{~d} t\leqslant C\label{dug1}.
\end{split}
\end{equation}
Moreover, choose $\delta=\delta_1=1/16$ in $\eqref{usj}$, hence $-\frac{36-15p}{2p}+3\delta_1<-1$, it follows from $\eqref{p1}$, $\eqref{xy}$, $\eqref{dbnlg}$ and $\eqref{dbsj}$ that
$$
\begin{aligned}
\int_0^TR_1(t)\mathrm{~d} t&=\int_0^T \|\nabla u\|_{L^2}^{3/2}\|\nabla B\|_{L^2}^2\mathrm{~d} t\\
&\leqslant \sup_{t\in[0,T]}\|\nabla B\|_{L^2}^2\bigg[\int_0^1 \|\nabla u\|_{L^2}^{3/2}\mathrm{~d} t\\
&\quad+\left(\int_1^T t^{\frac{12-5p}{2p}-\delta_1}\|\nabla u\|_{L^2}^2\mathrm{~d} t\right)^{3/4}\left(\int_1^Tt^{-\frac{36-15p}{2p}+3\delta_1}\mathrm{~d} t\right)^{1/4}\bigg]\leqslant C\|\nabla B_0\|_{L^2}^2,
\end{aligned}
$$
$$
\begin{aligned}
\int_0^TR_2(t)\mathrm{~d} t=\int_0^T \|\nabla B\|_{L^2}^6\mathrm{~d} t\leqslant \sup_{t\in[0,T]}\|\nabla B\|_{L^2}^4\int_0^T \|\nabla B\|_{L^2}^2\mathrm{~d} t\leqslant C\|\nabla B_0\|_{L^2}^2,\quad\quad\quad
\end{aligned}
$$
$$
\begin{aligned}
\int_0^TR_3(t)\mathrm{~d} t
&=\int_0^T \|B_t\|_{L^2}\|\nabla B\|_{L^2}^{1/2}\|\Delta B\|_{L^2}^{1/2}\|\nabla u\|_{L^2}\mathrm{~d} t\\
&\leqslant \sup_{t\in[0,T]}\|\nabla B\|_{L^2}^{1/2}\left(\int_0^T \|B_t\|_{L^2}^2\mathrm{~d} t\right)^{1/2}\left(\int_0^T \|\Delta B\|_{L^2}^2\mathrm{~d} t\right)^{1/4}\left(\int_0^T \|\nabla u\|_{L^2}^4\mathrm{~d} t\right)^{1/4}\\
& \leqslant C\|\nabla B_0\|_{L^2}^2,
\end{aligned}
$$
and
$$
\begin{aligned}
\int_0^TR_4(t)\mathrm{~d} t&=\int_0^T \|B_t\|_{L^2}\|\nabla B\|_{L^2}\|\nabla u\|_{L^2}\mathrm{~d} t\\
&\leqslant \sup_{t\in[0,T]}\|\nabla B\|_{L^2}\left(\int_0^T \|B_t\|_{L^2}^2\mathrm{~d} t\right)^{1/2}\left(\int_0^T \|\nabla u\|_{L^2}^2\mathrm{~d} t\right)^{1/2}\leqslant C\|\nabla B_0\|_{L^2}^2.
\end{aligned}
$$
It follows from $\eqref{ubj}$ and $\eqref{ydxw}$ that
$$
\|\nabla u\|_{L^2}^2\leqslant C_0\|\nabla u+(\nabla u)^T\|_{L^2}^2,
$$
which combining with $\eqref{G1}$, $\eqref{xy}$ and $\eqref{dbnlg}$, $J_0$ is estimated by
\begin{align}
\sup_{t\in[0,T]}|J_0|&=\sup_{t\in[0,T]}\left|\int B\cdot\nabla B\cdot u \mathrm{~d}x\right|\nonumber\\
&\leqslant\frac{\underline{\mu}}{16C_0}\sup_{t\in[0,T]}\|\nabla u\|_{L^2}^2+C\sup_{t\in[0,T]}\|B\|_{L^2}\|\nabla B\|_{L^2}^3\nonumber\\
&\leqslant\frac{\underline{\mu}}{16C_0}\sup_{t\in[0,T]}\|\nabla u\|_{L^2}^2+C\|\nabla B_0\|_{L^2}^2.\nonumber
\end{align}
Thus, integrating $\eqref{duhb}$ over (0,T) and using estimate $J_0$, $\eqref{dug1}$, $R_1$-$R_4$, we reduce that
\begin{align}\label{dugw}
\sup _{t \in[0,T]} \|\nabla u\|_{L^2}^2+\int_0^T \|\rho^{1 / 2} u_t\|_{L^2}^2\mathrm{~d} t \leqslant C\left(\|\nabla u_0\|_{L^2}^2+\|\nabla B_0\|_{L^2}^2\right) .
\end{align}

In particular, multiplying $\eqref{duhb}$ by $t^{\frac{6-2p}{p}-\delta}$, we get
\begin{align}\label{afdu}
&\frac{\mathrm{d}}{\mathrm{d} t} \int t^{\frac{6-2p}{p}-\delta}\mu(\rho)\left|\nabla u+(\nabla u)^T\right|^2 \mathrm{d} x+t^{\frac{6-2p}{p}-\delta}\|\rho^{1 / 2} u_t\|_{L^2}^2 \nonumber\\
&\leqslant \frac{\mathrm{d}}{\mathrm{d} t} \left(t^{\frac{6-2p}{p}-\delta}J_0\right)+C_1\left(\|\nabla u\|_{L^2}^4+\|\nabla B\|_{L^2}\|\Delta B\|_{L^2}\right)t^{\frac{6-2p}{p}-\delta}\|\nabla u\|_{L^2}^2\nonumber\\
&\quad+C\sum_{i=1}^4t^{\frac{6-2p}{p}-\delta}R_i(t)+Ct^{\frac{6-3p}{p}-\delta}\left(\|\nabla u\|_{L^2}^2+|J_0|\right).
\end{align}
Similar to the above estimate of $R_1$-$R_4$, combining with $\eqref{xy}$, $\eqref{dbnlg}$, $\eqref{dbsj}$ and $\eqref{usj}$, gives
$$
\begin{aligned}
&\int_0^Tt^{\frac{6-2p}{p}-\delta}R_1(t)\mathrm{~d} t\\
&=\int_0^T t^{\frac{6-2p}{p}-\delta}\|\nabla u\|_{L^2}^{3/2}\|\nabla B\|_{L^2}^2\mathrm{~d} t\\
&\leqslant\sup_{t\in[0,T]} t^{\frac{3+p}{4p}-\frac{\delta}{4}}\|\nabla B\|_{L^2}^{3/2}\left(\int_0^T t^{\frac{6-3p}{2p}-\delta}\|\nabla B\|_{L^2}^2\mathrm{~d} t\right)^{1/4}\left(\int_0^T t^{\frac{12-5p}{2p}-\frac{2\delta}{3}}\|\nabla u\|_{L^2}^2\mathrm{~d} t\right)^{3/4}\\
& \leqslant C(\delta),
\end{aligned}
$$
$$
\begin{aligned}
&\int_0^Tt^{\frac{6-2p}{p}-\delta}R_2(t)\mathrm{~d} t\\
&=\int_0^T t^{\frac{6-2p}{p}-\delta}\|\nabla B\|_{L^2}^6\mathrm{~d} t\\
&\leqslant \sup_{t\in[0,1]}\left(t^{\frac{6-2p}{p}-\delta}\|\nabla B\|_{L^2}^4\right)\int_0^1\|\nabla B\|_{L^2}^2\mathrm{~d} t+\sup_{t\in[1,T]}\left(t^{\frac{6-p}{2p}-\frac{\delta}{3}}\|\nabla B\|_{L^2}^2\right)^3\int_1^T t^{-\frac{6+p}{2p}}\mathrm{~d} t\quad\\
& \leqslant C(\delta),
\end{aligned}
$$
$$
\begin{aligned}
&\int_0^Tt^{\frac{6-2p}{p}-\delta}R_3(t)\mathrm{~d} t\\
&=\int_0^T t^{\frac{6-2p}{p}-\delta}\|B_t\|_{L^2}\|\nabla B\|_{L^2}^{1/2}\|\Delta B\|_{L^2}^{1/2}\|\nabla u\|_{L^2}\mathrm{~d} t\\
&\leqslant \sup_{t\in[0,T]}t^{\frac{6-2p}{2p}-\frac{\delta}{2}}\|\nabla u\|_{L^2}\left(\int_0^T t^{\frac{6-p}{2p}-\frac{\delta}{2}}\|B_t\|_{L^2}^2\mathrm{~d} t\right)^{1/2}\left(\int_0^T t^{\frac{6-3p}{2p}-\frac{\delta}{2}}\|\nabla B\|_{L^2}^2\mathrm{~d} t\right)^{1/4}\quad\quad\\
&\quad\quad\quad \ \times\left(\int_0^T t^{\frac{6-3p}{2p}-\frac{\delta}{2}}\|\Delta B\|_{L^2}^2\mathrm{~d} t\right)^{1/4}\\
& \leqslant \frac{\underline{\mu}}{16C_0}\sup_{t\in[0,T]}t^{\frac{6-2p}{p}-\delta}\|\nabla u\|_{L^2}^2+C(\delta),
\end{aligned}
$$
and
$$
\begin{aligned}
&\int_0^Tt^{\frac{6-2p}{p}-\delta}R_4(t)\mathrm{~d} t\\
&=\int_0^T t^{\frac{6-2p}{p}-\delta}\|B_t\|_{L^2}\|\nabla B\|_{L^2}\|\nabla u\|_{L^2}\mathrm{~d} t\\
&\leqslant C\sup_{t\in[0,T]}t^{\frac{6-p}{4p}-\frac{\delta}{2}}\|\nabla B\|_{L^2}\left(\int_0^Tt^{\frac{6-p}{2p}-\frac{\delta}{2}}\|B_t\|_{L^2}^2\mathrm{~d} t\right)^{1/2}\left(\int_0^T t^{\frac{12-5p}{2p}-\frac{\delta}{2}}\|\nabla u\|_{L^2}^2\mathrm{~d} t\right)^{1/2}\quad\quad\\
& \leqslant C(\delta).
\end{aligned}
$$
It follows from $\eqref{xy}$, $\eqref{dbnlg}$, $\eqref{dbsj}$ and $\eqref{usj}$, $J_0$ is estimated by
\begin{align}
&\sup_{t\in[0,T]}t^{\frac{6-2p}{p}-\delta}|J_0|\nonumber\\
&\leqslant\frac{\underline{\mu}}{16C_0}\sup_{t\in[0,T]}t^{\frac{6-2p}{p}-\delta}\|\nabla u\|_{L^2}^2+C\sup_{t\in[0,T]}t^{\frac{6-2p}{p}-\delta}\|B\|_{L^2}\|\nabla B\|_{L^2}^3\nonumber\\
&\leqslant\frac{\underline{\mu}}{16C_0}\sup_{t\in[0,T]}t^{\frac{6-2p}{p}-\delta}\|\nabla u\|_{L^2}^2+C\sup_{t\in[0,T]}t^{\frac{6-3p}{4p}-\frac{\delta}{2}}\|B\|_{L^2}\sup_{t\in[0,T]}t^{\frac{18-5p}{4p}-\frac{\delta}{2}}\|\nabla B\|_{L^2}^3\nonumber\\
&\leqslant \frac{\underline{\mu}}{16C_0}\sup_{t\in[0,T]}t^{\frac{6-2p}{p}-\delta}\|\nabla u\|_{L^2}^2+C(\delta),\nonumber
\end{align}
and
$$
\begin{aligned}
&\int_0^Tt^{\frac{6-3p}{p}-\delta}\left(\|\nabla u\|_{L^2}^2+|J_0|\right)\mathrm{~d} t\\
&\leqslant C\int_0^T t^{\frac{6-3p}{p}-\delta}\left(\|\nabla u\|_{L^2}^2+\|B\|_{L^2}\|\nabla B\|_{L^2}^3\right)\mathrm{~d} t\\
&\leqslant C(\delta)+C\sup_{t\in[0,T]}t^\frac{6-3p}{4p}\|B\|_{L^2}\sup_{t\in[0,T]}t^{\frac{6-3p}{4p}-\frac{\delta}{2}}\|\nabla B\|_{L^2}\int_0^Tt^{\frac{6-3p}{2p}-\frac{\delta}{2}}\|\nabla B\|_{L^2}^2\mathrm{~d} t\\
&\leqslant C(\delta).
\end{aligned}
$$
Hence, integrating $\eqref{afdu}$ over (0,T) and using estimate $J_0$, $\eqref{dug1}$, $R_1$-$R_4$, we reduce that
\begin{align}
\sup _{t \in[0,T]}t^{\frac{6-2p}{p}-\delta}\|\nabla u\|_{L^2}^2+\int_0^T t^{\frac{6-2p}{p}-\delta}\|\rho^{1 / 2} u_t\|_{L^2}^2 \mathrm{~d} t \leqslant C(\delta) .\label{thu2g}
\end{align}
The proof of Lemma 3.5 is finished.
\end{proof}

\begin{lemma}
Let $(\rho, u, B, P)$ be a smooth solution to $\eqref{MHD}$-$\eqref{ydxw}$ satisfying $\eqref{p1}$. Then there exists a generic positive constant $C$ depending only on $p, q, \bar{\rho}, \underline{\mu}, \bar{\mu}, \nu, \Omega, \|\rho_0\|_{L^{3/2}}$, $\|B_0\|_{L^{p}}$ and $M$ such that for any $0<\delta<1/8$, one obtain
\begin{align}
& \sup_{t \in[0,T]} t\left(\|\rho^{1 / 2} u_t\|_{L^2}^2+\|B_t\|_{L^2}^2+\|\Delta B\|_{L^2}^2\right)+\int_0^T t\left(\|\nabla u_t\|_{L^2}^2+\|\nabla B_t\|_{L^2}^2\right) \mathrm{~d} t \nonumber\\
&\leqslant C\left(\|\nabla u_0\|_{L^2}^2+\|\nabla B_0\|_{L^2}^2\right),\label{ddb}\\
&\sup _{t \in[0, T]} t^{\frac{6-p}{p}-\delta}\|\rho^{1 / 2} u_t\|_{L^2}^2+\int_0^T t^{\frac{6-p}{p}-\delta}\|\nabla u_t\|_{L^2}^2 \mathrm{~d} t \leqslant C(\delta) ,\label{t3mg}\\
&\sup _{t \in[0, T]} t^{\frac{6+p}{2p}-\delta}\left(\|B_t\|_{L^2}^2+\|\Delta B\|_{L^2}^2\right)+\int_0^T t^{\frac{6+p}{2p}-\delta}\|\nabla B_t\|_{L^2}^2 \mathrm{~d} t \leqslant C(\delta),\label{t2mg}
\end{align}
provided $\|\nabla u_0\|_{L^2}+||\nabla B_0\|_{L^2}\leqslant 1$.
\end{lemma}
\begin{proof}
First, operating $\partial_t$ to $\eqref{MHD}_2$ yields that
\begin{equation}\label{tmiu}
\begin{split}
& \rho u_{tt}+\rho u \cdot \nabla u_t-\operatorname{div}\left(2 \mu(\rho)\cdot d_t\right)+\nabla P_t \\
&=-\rho_t u_t-(\rho u)_t \cdot \nabla u+\operatorname{div}\left(2\mu(\rho)_t\cdot d\right)+(B \cdot \nabla B)_t .
\end{split}
\end{equation}

On the one hand, multiplying the above equality by $u_t$, we obtain after using integration by parts and $\eqref{MHD}_1$ that
\begin{equation}\label{tmus}
\begin{split}
& \frac{1}{2} \frac{\mathrm{d}}{\mathrm{d} t} \int \rho|u_t|^2 \mathrm{~d} x+2 \int \mu(\rho)|d_t|^2 \mathrm{~d} x \\
&=-2 \int \rho u \cdot \nabla u_t \cdot u_t \mathrm{~d} x-\int \rho u \cdot \nabla\left(u \cdot \nabla u \cdot u_t\right) \mathrm{d} x-\int \rho u_t \cdot \nabla u \cdot u_t \mathrm{~d} x \\
&\quad+\int 2 u \cdot \nabla \mu(\rho) d \cdot \nabla u_t \mathrm{~d} x+\int(B \cdot \nabla B)_t \cdot u_t \mathrm{~d} x \\
&\triangleq \sum_{i=1}^5 K_i .
\end{split}
\end{equation}
Now, we will use $\eqref{G1}$ and $\eqref{midu}$ to estimate each term on the right hand of $\eqref{tmus}$ as follows:
$$
\begin{aligned}
|K_1|+|K_3| & \leqslant C\|\rho^{1/2} u_t\|_{L^3}\|\nabla u_t\|_{L^2}\|u\|_{L^6}+C\|\rho^{1/2} u_t\|_{L^3}\|\nabla u\|_{L^2}\|u_t\|_{L^6} \nonumber\\
& \leqslant C\|\rho^{1/2} u_t\|_{L^2}^{1/2}\|\nabla u_t\|_{L^2}^{3/2}\|\nabla u\|_{L^2} \\
& \leqslant \frac{1}{8} \underline{\mu}\|\nabla u_t\|_{L^2}^2+C\|\rho^{1/2} u_t\|_{L^2}^4+C\|\nabla u\|_{L^2}^8 .\nonumber
\end{aligned}
$$
Next by Lemma 2.2, $\eqref{midu}$, $\eqref{mmdq}$ and $\eqref{tht}$
$$
\begin{aligned}
|K_2|
&\leqslant C\int \rho |u||u_t|\left(|\nabla u|^2+|u||\nabla^2 u|\right)\mathrm{~d} x+\int\rho|u|^2|\nabla u||\nabla u_t|\mathrm{~d} x \nonumber\\
&\leqslant C\|u\|_{L^6}\|u_{t}\|_{L^6}\left(\|\nabla u\|_{L^3}^2+\|u\|_{L^6}\|\nabla^2 u\|_{L^2}\right)+C\|u\|_{L^6}^2\|\nabla u\|_{L^6}\|\nabla u_t\|_{L^2} \nonumber\\
& \leqslant C\|\nabla u_t\|_{L^2}\|\nabla^2 u\|_{L^2}\|\nabla u\|_{L^2}^2 \\
& \leqslant \frac{1}{8} \underline\mu\|\nabla u_t\|_{L^2}^2+C\|\nabla^2 u\|_{L^2}^2\|\nabla u\|_{L^2}^4 \nonumber\\
& \leqslant \frac{1}{8} \underline\mu\|\nabla u_t\|_{L^2}^2+C\|\rho^{1/2} u_t\|_{L^2}^4+C\|\nabla u\|_{L^2}^{10}+C\|\nabla u\|_{L^2}^8+C\|B\|_{L^2}^2\|\nabla B\|_{L^2}^6\nonumber\\
&\quad+C\|\nabla B\|_{L^2}^8+C\|\nabla B\|_{L^2}^6\|\Delta B\|_{L^2}^2, \nonumber
\end{aligned}
$$
$$
\begin{aligned}
|K_4|
&\leqslant C\|u\|_{L^\infty}\|\nabla\mu(\rho)\|_{L^q}\|\nabla u\|_{L^{2q/(q-2)}}\|\nabla u_t\|_{L^2}\nonumber\\
&\leqslant C\|u\|_{L^6}^{1/2}\|\nabla u\|_{L^6}^{1/2}\|\nabla u\|_{L^2}^{(q-3)/q}\|\nabla^2 u\|_{L^2}^{3/q}\|\nabla u_t\|_{L^2}\nonumber\\
& \leqslant \frac{1}{8} \underline\mu\|\nabla u_t\|_{L^2}^2+C\|\nabla u\|_{L^2}^4+\|\nabla u\|_{L^2}\|\nabla^2 u\|_{L^2}^3 \\
& \leqslant \frac{1}{8} \underline\mu\|\nabla u_t\|_{L^2}^2+C\|\nabla u\|_{L^2}^4+C\|\rho^{1 / 2} u_t\|_{L^2}^4+C\|\nabla u\|_{L^2}^{10}+C\|B\|_{L^2}^2\|\nabla B\|_{L^2}^6\nonumber\\
&\quad+C\|\nabla B\|_{L^2}^8+C\|\nabla B\|_{L^2}^6\|\Delta B\|_{L^2}^2,\nonumber
\end{aligned}
$$
and
$$
\begin{aligned}
|K_5| \leqslant \frac{1}{8} \underline\mu\|\nabla u_t\|_{L^2}^2+C\|\nabla B\|_{L^2}\|\Delta B\|_{L^2}\|B_t\|_{L^2}^2+C\|\nabla B\|_{L^2}^2\|B_t\|_{L^2}^2 .
\end{aligned}
$$
Substituting $K_1$, $K_2$, $K_3$, $K_4$ and $K_5$ into $\eqref{tmus}$ gives
\begin{align}\label{tmuh}
& \frac{\mathrm{d}}{\mathrm{d} t}\|\rho^{1 / 2} u_t\|_{L^2}^2+\underline\mu\|\nabla u_t\|_{L^2}^2\nonumber\\
& \leqslant C\|\rho^{1 / 2} u_t\|_{L^2}^4+C\|\nabla u\|_{L^2}^4+C\|\nabla B\|_{L^2}^6\|\Delta B\|_{L^2}^2+C\|\nabla B\|_{L^2}^6\\
& \quad+C\|\nabla B\|_{H^1}^2\|B_t\|_{L^2}^2.\nonumber
\end{align}
Hence, multiplying $\eqref{tmuh}$ by $t$, integrating it over $(0,T)$, and taking $\eqref{xy}$, $\eqref{dbnlg}$, $\eqref{dbnlt}$, $\eqref{usj}$ and $\eqref{dux}$ into account, we immediately obtain
\begin{align}\label{utg3}
\sup_{t \in[0,T]} t\|\rho^{1 / 2} u_t\|_{L^2}^2+\int_0^T t\|\nabla u_t\|_{L^2}^2 \mathrm{~d} t\leqslant C\left(\|\nabla u_0\|_{L^2}^2+\|\nabla B_0\|_{L^2}^2\right).
\end{align}
Furthermore, we multiplying $\eqref{tmuh}$ by $t^{\frac{6-p}{p}-\delta}$, which together with $\eqref{dbnlg}$, $\eqref{dbsj}$, $\eqref{usj}$, $\eqref{dux}$ and $\eqref{tdugj}$ yield
\begin{align}\label{putt}
&\sup _{t \in[0, T]} t^{\frac{6-p}{p}-\delta}\|\rho^{1 / 2} u_t\|_{L^2}^2+\int_0^T t^{\frac{6-p}{p}-\delta}\|\nabla u_t\|_{L^2}^2 \mathrm{~d} t \leqslant C(\delta) .
\end{align}

On the other hand, operating $\partial_t$ to $\eqref{MHD}_3$ yields that
\begin{align}\label{btt}
B_{t t}-\nu \Delta B_t=(B \cdot \nabla u-u \cdot \nabla B)_t,
\end{align}
multiplying $\eqref{btt}$ by $B_t$, we obtain after using integration by parts, $\eqref{dvgja}$ and the H\"older's inequality that
\begin{align}\label{bttg}
&\frac{\mathrm{d}}{\mathrm{d} t}\|B_t\|_{L^2}^2+2\nu\|\operatorname{curl} B_t\|_{L^2}^2 \nonumber\\
& \leqslant C\|B_t\|_{L^3}\|\nabla u\|_{L^2}\|B_t\|_{L^6}+C\|B\|_{L^\infty}\|\nabla u_t\|_{L^2}\|B_t\|_{L^2}+C\|u_t\|_{L^6}\|\nabla B\|_{L^3}\|B_t\|_{L^2} \nonumber\\
& \leqslant C\|\nabla u_t\|_{L^2}^2+C\left(\|\nabla u\|_{L^2}^4+\|\nabla B\|_{H^1}^2\right)\|B_t\|_{L^2}^2.
\end{align}
It follows from Eq. $\eqref{MHD}_3$ and Lemma 2.2 that
\begin{align}
\|\Delta B\|_{L^2}^2 \leqslant C\|B_t\|_{L^2}^2+C\|\nabla u\|_{L^2}^4\|\nabla B\|_{L^2}^2 .
\end{align}
Hence, multiplying $\eqref{bttg}$ by $t$, integrating it over $(0,T)$, and taking $\eqref{p1}$, $\eqref{xy}$, $\eqref{dbnlg}$ and $\eqref{utg3}$ into account, we immediately obtain
$$
\sup_{t \in[0,T]} t\left(\|B_t\|_{L^2}^2+\|\Delta B\|_{L^2}^2\right)+\int_0^T t\|\nabla B_t\|_{L^2}^2 \mathrm{~d} t\leqslant C\left(\|\nabla u_0\|_{L^2}^2+\|\nabla B_0\|_{L^2}^2\right).
$$
Furthermore, multiplying $\eqref{bttg}$ by $t^{\frac{6+p}{2p}-\delta}$, which together with $\eqref{p1}$, $\eqref{xy}$, $\eqref{dbnlg}$, $\eqref{dbsj}$, $\eqref{usj}$, $\eqref{dux}$, $\eqref{tdugj}$ and $\eqref{putt}$ yield
$$
\begin{aligned}
\sup _{t \in[0, T]} t^{\frac{6+p}{2p}-\delta}\left(\|B_t\|_{L^2}^2+\|\Delta B\|_{L^2}^2\right)+\int_0^T t^{\frac{6+p}{2p}-\delta}\|\nabla B_t\|_{L^2}^2 \mathrm{~d} t \leqslant C(\delta).
\end{aligned}
$$
The proof of Lemma 3.6 is finished.
\end{proof}

\begin{lemma}
Let $(\rho, u, B, P)$ be a smooth solution to $\eqref{MHD}$-$\eqref{ydxw}$ satisfying $\eqref{p1}$. Then there exists a generic positive constant $C$ depending only on $p, q, \overline{\rho}, \underline\mu, \overline{\mu}, \nu, \Omega, \|\rho_0\|_{L^{3/2}}$, $\|B_0\|_{L^{p}}$ and $M$ such that for $\beta\triangleq\frac{12-7p}{12-2p}$, one obtain
\begin{align}\label{duwq}
\int_0^T\|\nabla u\|_\infty \mathrm{~d} t \leqslant C\left(\|\nabla u_0\|_{L^2}+\|\nabla B_0\|_{L^2}\right)^{\beta},
\end{align}
provided $\|\nabla u_0\|_{L^2}+||\nabla B_0\|_{L^2}\leqslant 1$.
\end{lemma}
\begin{proof}
First, it follows from Lemma $2.2$ that for any $\alpha \in[2, \min \{6, q\}]\cap[2,6)$,
\begin{align}
& \|\rho u_t+\rho u \cdot \nabla u-B \cdot \nabla B\|_{L^\alpha} \nonumber\\
& \leqslant C\|\rho^{1 / 2} u_t\|_{L^2}^{(6-\alpha) / 2 \alpha}\|\rho^{1 / 2} u_t\|_{L^6}^{(3\alpha-6)/ 2 \alpha}+C\|u\|_{L^6}\|\nabla u\|_{L^{6\alpha/(6-\alpha)}}+C\|B\|_{L^\infty}\|\nabla B\|_{L^\alpha} \nonumber\\
& \leqslant C\|\rho^{1 / 2} u_t\|_{L^2}^{(6-\alpha) / 2 \alpha}\|\nabla u_t\|_{L^2}^{(3 \alpha-6) / 2 \alpha}+C\|\nabla u\|_{L^2}^{6(\alpha-1) /(5 \alpha-6)}\|\nabla^2 u\|_{L^\alpha}^{(4 \alpha-6) /(5 \alpha-6)} \nonumber\\
& \quad+C\|\nabla B\|_{L^2}^{3 / \alpha}\|\nabla^2 B\|_{L^2}^{(2 \alpha-3) / \alpha} \nonumber\\
& \leqslant \varepsilon\|\nabla^2 u\|_{L^\alpha}+C\|\rho^{1 / 2} u_t\|_{L^2}^{(6-\alpha) / 2 \alpha}\|\nabla u_t\|_{L^2}^{(3\alpha-6) / 2 \alpha}+C_\varepsilon\|\nabla u\|_{L^2}^{6(\alpha-1) / \alpha} \nonumber\\
& \quad+C\|\nabla B\|_{L^2}^2+C\|\nabla^2 B\|_{L^2}^2,\nonumber
\end{align}
which together with $\eqref{tht}$ gives
\begin{align}\label{du2p}
&\|\nabla^2 u\|_{L^\alpha}+\|\nabla u\|_{L^2}+\|\nabla P\|_{L^\alpha} \nonumber\\
&\leqslant C\|\rho u_t+\rho u \cdot \nabla u-B \cdot \nabla B\|_{L^{6 / 5}\cap L^\alpha}\nonumber\\
& \leqslant C\|\rho^{1 / 2} u_t\|_{L^2}^{(6-\alpha) / 2 \alpha}\|\nabla u_t\|_{L^2}^{(3 \alpha-6) / 2 \alpha}+C\|\nabla u\|_{L^2}^3+C\|\rho^{1 / 2} u_t\|_{L^2} +C\|\Delta B\|_{L^2}^2\\
& \quad+C\|\nabla B\|_{L^2}^2+C\|B\|_{L^2}^{1 / 2}\|\nabla B\|_{L^2}^{3 / 2} ,\nonumber
\end{align}
where we use $\frac{6(\alpha-1)}{\alpha}\geqslant 3$.

Then, setting
$$
r \triangleq \frac{1}{2} \min \left\{q+3, 9\right\} \in\left(3, \min \{q, 6\}\right),
$$
one derives from the $\eqref{G1}$ and $\eqref{du2p}$ that
\begin{align}\label{duwqg}
\|\nabla u\|_{L^\infty} & \leqslant C\|\nabla u\|_{L^2}+C\|\nabla^2 u\|_{L^r} \nonumber\\
& \leqslant C\|\rho^{1 / 2} u_t\|_{L^2}^{(6-r) / 2 r}\|\nabla u_t\|_{L^2}^{(3 r-6) / 2 r}+C\|\nabla u\|_{L^2}^3+C\|\rho^{1 / 2} u_t\|_{L^2} \\
&\quad+C\|\Delta B\|_{L^2}^2+C\|\nabla B\|_{L^2}^2+C\|B\|_{L^2}^{1 / 2}\|\nabla B\|_{L^2}^{3 / 2} .\nonumber
\end{align}

On the one hand, it follows from $\eqref{xy}$, $\eqref{dbnlg}$, $\eqref{dux}$ and $\eqref{ddb}$ that for $t \in(0,1]$,
$$
\begin{aligned}
\|\nabla u\|_{L^\infty} \leqslant & C\left(\|\nabla u_0\|_{L^2}+\|\nabla B_0\|_{L^2}\right)^{(6-r) / 2 r} t^{-1/2}\left(t\|\nabla u_t\|_{L^2}^2\right)^{(3 r-6) / 4r} \\
& +C\left(\|\nabla u_0\|_{L^2}+\|\nabla B_0\|_{L^2}\right)^3 +C\left(\|\nabla u_0\|_{L^2}+\|\nabla B_0\|_{L^2}\right) t^{-1/2} \\
& +C\|\Delta B\|_{L^2}^2+C\left(\|\nabla u_0\|_{L^2}+\|\nabla B_0\|_{L^2}\right)^2+C\left(\|\nabla u_0\|_{L^2}+\|\nabla B_0\|_{L^2}\right)^{3/2},
\end{aligned}
$$
which together with $\eqref{dbnlg}$ and $\eqref{ddb}$ gives
\begin{align}\label{xjduw}
& \int_0^1\|\nabla u\|_{L^\infty} \mathrm{~d} t \nonumber\\
& \leqslant\left(\|\nabla u_0\|_{L^2}+\|\nabla B_0\|_{L^2}\right)^{(6-r)/2r}\left(\int_0^1 t^{-2r/(r+6)} \mathrm{~d} t\right)^{(r+6) / 4 r}\left(\int_0^1 t\|\nabla u_t\|_{L^2}^2\mathrm{~d} t\right)^{(3r-6)/4r} \nonumber\\
& \quad+C\left(\|\nabla u_0\|_{L^2}+\|\nabla B_0\|_{L^2}\right) \nonumber\\
& \leqslant C\left(\|\nabla u_0\|_{L^2}+\|\nabla B_0\|_{L^2}\right).
\end{align}

On the other hand, using the H\"older inequality and $\eqref{midu}$, we obtain that for $t \in[1, T]$,
$$
\begin{aligned}
\|\nabla u\|_{L^\infty} &\leqslant C\|\rho\|_{L^{3 / 2}}^{(6-r) / 4r}\|\nabla u_t\|_{L^2}+C\|\nabla u\|_{L^2}^3 +C\|\rho\|_{L^{3 / 2}}\|\nabla u_t\|_{L^2}+C\|\Delta B\|_{L^2}^2\\
&\quad+C\|\nabla B\|_{L^2}^2+C\|B\|_{L^2}^{1 / 2}\|\nabla B\|_{L^2}^{3 / 2} \\
& \leqslant C\|\nabla u_t\|_{L^2}+C\|\nabla u\|_{L^2}^3+C\|\Delta B\|_{L^2}^2+C\|\nabla B\|_{L^2}^2+C\|B\|_{L^2}^{1 / 2}\|\nabla B\|_{L^2}^{3 / 2},
\end{aligned}
$$
and thus
\begin{align}\label{dduwq}
&\int_1^T\|\nabla u\|_{L^\infty}\mathrm{d} t \nonumber\\
&\leqslant C \int_1^T\left(\|\nabla u_t\|_{L^2}+\|\nabla u\|_{L^2}^3+\|\Delta B\|_{L^2}^2+\|\nabla B\|_{L^2}^2+\|B\|_{L^2}^{1/2}\|\nabla B\|_{L^2}^{3/2}\right) \mathrm{~d} t,
\end{align}
choose $\delta=\delta_2=\frac{1}{16}$ in $\eqref{t3mg}$, hence $-\frac{3}{p}+\frac{\delta_2}{2}={-\frac{96-p}{32p}}<-1$, and using $\eqref{xy}$, $\eqref{dbnlg}$, $\eqref{dux}$ and $\eqref{ddb}$, we get by direct calculations that
\begin{align}
& \int_1^T\left(\|\nabla u_t\|_{L^2}+\|\nabla u\|_{L^2}^3+\|\Delta B\|_{L^2}^2\right) \mathrm{~d} t \nonumber\\
& \leqslant \left(\int_1^Tt\|\nabla u_t\|_{L^2}^2 \mathrm{~d} t\right)^{1 / 4}\left(\int_1^T t^{\frac{6-p}{p}-\delta_2}\|\nabla u_t\|_{L^2}^2 d t\right)^{1 / 4}\left(\int_1^T t^{-\frac{3}{p}+\frac{\delta_2}{2}} \mathrm{~d} t\right)^{1 / 2} \nonumber\\
&\quad +\sup_{t\in[1,T]}\|\nabla u\|_{L^2}\int_1^T\|\nabla u\|_{L^2}^2 \mathrm{~d} t+\int_1^T\|\Delta B\|_{L^2}^2\mathrm{~d} t\label{djduw}\\
& \leqslant C\left(\|\nabla u_0\|_{L^2}+\|\nabla B_0\|_{L^2}\right)^{1 / 2},\nonumber
\end{align}
choose $\delta=\delta_3=\min\{\frac{(6-p)(12-7p)}{4p(12+3p)},\frac{1}{8}\}$ in $\eqref{dbsj}$, hence ${-\frac{12+p}{8p}+\frac{12+3p}{4(6-p)}\delta_3}={-\frac{12+9p}{16p}}<-1$,
and using $\eqref{dbnlg}$, $\eqref{dbsj}$ and $\eqref{usj}$, one obtain
\begin{align}
& \int_1^T\left(\|\nabla B\|_{L^2}^2+\|B\|_{L^2}^{1 / 2}\|\nabla B\|_{L^2}^{3 / 2}\right) \mathrm{~d} t \nonumber\\
& \leqslant C\sup _{t \in[1, T]}\|\nabla B\|_{L^2}^\frac{12-7p}{12-2p}\sup _{t \in[1, T]} t^{\frac{6-3p}{8p}-\frac{\delta_3}{4}}\left(\| B\|_{L^2}^{1/2}+\|\nabla B\|_{L^2}^{1/2}\right)\nonumber\\
&\quad\times\sup_{t\in[1,T]}\left(t^{\frac{6-p}{2p}-\delta_3}\|\nabla B\|_{L^2}^2\right)^\frac{3+2p}{12-2p}\int_1^T t^{-\frac{12+p}{8p}+\frac{12+3p}{4(6-p)}\delta_3} \mathrm{~d} t \nonumber\\
& \leqslant C\left(\|\nabla B_0\|_{L^2}\right)^\frac{12-7p}{12-2p},\nonumber
\end{align}
which together with $\eqref{xjduw}$-$\eqref{djduw}$ gives $\eqref{duwq}$ and finishes the proof of Lemma 3.7.
\end{proof}

With Lemma 3.2-3.7 at hand, we are in a position to prove Proposition 3.1.\\
\textbf{proof of proposition 3.1}
First, it from $\eqref{duwq}$ that
\begin{align}\label{uwqg1}
\int_0^T\|\nabla u\|_\infty \mathrm{~d} t \leqslant C_1\left(\|\nabla u_0\|_{L^2}+\|\nabla B_0\|_{L^2}\right)^{\beta} \leqslant \frac{\ln 2}{q},
\end{align}
provided that
$$
\|\nabla u_0\|_{L^2}+\|\nabla B_0\|_{L^2} \leqslant \varepsilon_1 \triangleq \left\{1,\left(\frac{\ln 2}{qC_1}\right)^{\frac{1}{\beta}}\right\},
$$
similar to $\eqref{mmdq}$, which also can deduce
$$
\sup _{0 \leqslant t<\infty}\|\nabla \mu(\rho)\|_{L^q} \leqslant 2\|\nabla \mu(\rho_0)\|_{L^q}.
$$

Next, choose $\delta=\delta_4=\min\{\frac{7-4p}{16p},1/8\}$ in $\eqref{dbsj}$, hence ${(\frac{\delta_4}{2}-\frac{6-p}{4p})\frac{16p}{17}}={\frac{4p-41}{34}}<-1$, it follows from $\eqref{xy}$ and $\eqref{dbnlg}$ that
\begin{align}\label{bpgp}
\int_0^T\|\nabla B\|_{L^2}^p\mathrm{~d} t&\leqslant\sup_{t\in[0,1]}\|\nabla B\|_{L^2}^\frac{p}{17}\int_0^1\|\nabla B\|_{L^2}^\frac{16p}{17}\mathrm{~d} t\nonumber\\
&\quad+\sup_{t\in[1,T]}\|\nabla B\|_{L^2}^\frac{p}{17}\Big(\sup_{t\in[1,T]}t^{(\frac{6-p}{4p}-\frac{\delta_4}{2})}\|\nabla B\|_{L^2}\Big)^\frac{16p}{17}\int_1^Tt^{(\frac{\delta_4}{2}-\frac{6-p}{4p})\frac{16p}{17}}\mathrm{~d} t\nonumber\\
&\leqslant C_2\|\nabla B_0\|_{L^2}^\frac{p}{17}\leqslant 1,
\end{align}
provided that
$$
\begin{aligned}
\|\nabla u_0\|_{L^2}+\|\nabla B_0\|_{L^2}\leqslant \varepsilon_2 \triangleq \left\{1,C_2^{-\frac{17}{p}}\right\}.
\end{aligned}
$$
In particular, it from $\eqref{tht}$ that
$$
\begin{aligned}
&\|\nabla u\|_{H^1}+\|P\|_{H^1} \\& \leqslant C\|\rho^{1 / 2} u_t\|_{L^2}+C\|\nabla u\|_{L^2}^3+C\|\nabla B\|_{L^2}^{3/2}\|\Delta B\|_{L^2}^{1/2}+C\|B\|_{L^2}^{1/2}\|\nabla B\|_{L^2}^{3/2},
\end{aligned}
$$
which, together with $\eqref{xy}$, $\eqref{dbnlg}$, $\eqref{dbsj}$, $\eqref{usj}$, $\eqref{dux}$, $\eqref{tdugj}$, $\eqref{ddb}$ and $\eqref{t3mg}$, gives
\begin{align}\label{du5g}
\sup_{t\in[0,T]} t^{\frac{12-3p}{2p}-\delta}\left(\|\nabla u\|_{H^1}^2+\|P\|_{H^1}^2\right)\leqslant C.
\end{align}

Then, choose $\delta=\delta_5=\min\{\frac{12-7p}{4p},1/8\}$ in $\eqref{du5g}$, hence $\frac{\delta_5}{2}-\frac{12-3p}{4p}\leqslant -\frac{12+p}{8p}<-1$, it follows from $\eqref{xy}$ and $\eqref{dux}$ that
\begin{equation}\label{4du}
\begin{split}
&\int_0^T\|\nabla u\|_{L^2}^4\mathrm{~d} t\\
&\leqslant \sup_{t\in[0,T]}\|\nabla u\|_{L^2}^3\left(\int_0^1\|\nabla u\|_{L^2}\mathrm{~d} t+\sup_{t\in[1,T]}t^{\frac{12-3p}{4p}-\frac{\delta_5}{2}}\|\nabla u\|_{L^2}\int_1^Tt^{\frac{\delta_5}{2}-\frac{12-3p}{4p}}\mathrm{~d} t\right)\\
&\leqslant C_3\left(\|\nabla u_0\|_{L^2}+\|\nabla B_0\|\right)\left(\|\nabla u_0\|_{L^2}^2+\|\nabla B_0\|_{L^2}^2\right)\\
&\leqslant \|\nabla u_0\|_{L^2}^2+\|\nabla B_0\|_{L^2}^2,
\end{split}
\end{equation}
provided that
$$
\begin{aligned}
\|\nabla u_0\|_{L^2}+\|\nabla B_0\|_{L^2}\leqslant \varepsilon_3 \triangleq \left\{1,C_3^{-1}\right\}.
\end{aligned}
$$

Choosing $\varepsilon_0 \triangleq \min \left\{\varepsilon_1, \varepsilon_2, \varepsilon_3 \right\}$, we directly obtain $\eqref{p2}$ from $\eqref{uwqg1}$, $\eqref{bpgp}$ and $\eqref{4du}$. The proof of Proposition $3.1$ is completed.
\hfill$\square$

The following Lemma 3.8-3.9 are concerned with further estimates on the higher-order derivatives of the strong solution $(\rho, u, B, P)$:
\begin{lemma}
Let $(\rho, u, B, P)$ be a smooth solution to $\eqref{MHD}$-$\eqref{ydxw}$ satisfying $\eqref{p1}$. Then there exists a generic positive constant $C$ depending only on $p, q, \bar{\rho}, \underline{\mu}, \bar{\mu}, \nu, \Omega, \|\rho_0\|_{L^{3/2}}$, $\|B_0\|_{L^{p}}$ and $M$ such that for any $0<\delta<1/8$, one obtain
\begin{align}\label{dut3}
\sup _{t \in[0, T]} t^{{\frac{6+3p}{2p}-\delta}}\|\nabla B_t\|_{L^2}^2+\int_0^T t^{{\frac{6+3p}{2p}-\delta}}\|B_{t t}\|_{L^2}^2+t^{{\frac{6+p}{2p}-\delta}}\|\nabla^2 B_t\|_{L^2}^2 \mathrm{~d} t \leqslant C(\delta),
\end{align}
provided $\|\nabla u_0\|_{L^2}+\|\nabla B_0\|_{L^2}\le \varepsilon_0$.
\end{lemma}
\begin{proof}
First, multiplying $\eqref{btt}$ by $B_{t t}$, we obtain after using integration by parts, the Cauchy-Schwarz's inequality and $\eqref{tht}$ that
\begin{align}\label{dut2b}
&\nu\frac{\mathrm{d}}{\mathrm{d} t}\|\operatorname{curl} B_t\|_{L^2}^2+\|B_{tt}\|_{L^2}^2 \nonumber\\
&\leqslant C\left(\|\nabla u\|_{H^1}^2+\|\nabla B\|_{H^1}^2\right)\left(\|\nabla u_t\|_{L^2}^2+\|\nabla B_t\|_{L^2}^2\right)\\
&\leqslant C\left(\|\rho^{1/2}u_t\|_{L^2}^2+\|\nabla u\|_{L^2}^6+\|\nabla B\|_{H^1}^2\right)\left(\|\nabla u_t\|_{L^2}^2+\|\nabla B_t\|_{L^2}^2\right) .\nonumber
\end{align}
Hence, multiplying $\eqref{dut2b}$ by $t^{{\frac{6+3p}{2p}-\delta}}$ and integrating the resulting inequality over $(0, T)$, we deduce from $\eqref{dvgja}$, $\eqref{xy}$, $\eqref{dbnlg}$, $\eqref{usj}$, $\eqref{dux}$, $\eqref{tdugj}$, $\eqref{ddb}$ and $\eqref{t2mg}$ that
\begin{align}\label{bt3g}
\sup _{t \in[0, T]} t^{{\frac{6+3p}{2p}-\delta}}\|\nabla B_t\|_{L^2}^2+\int_0^T t^{{\frac{6+3p}{2p}-\delta}}\|B_{t t}\|_{L^2}^2 \mathrm{~d} t \leqslant C(\delta) .
\end{align}

Next, operating $\partial_t$ to $\eqref{MHD}_3$ yields that
$$
B_{tt}+(u_t\cdot\nabla B+u\cdot\nabla B_t)-(B_t\cdot\nabla u+B\cdot\nabla u_t)-\nu\Delta B_t=0.
$$
Hence, one can deduce from Lemma 2.5 and Sobolev inequality that
$$
\begin{aligned}
&\|\nabla^2 B_t\|_{L^2}^2 \\
&\leqslant C\|\Delta B_t\|_{L^2}^2+C\|\nabla B_t\|_{L^2}^2\\
&\leqslant C\|B_{t t}\|_{L^2}^2+C\|u_t\|_{L^6}^2\|\nabla B\|_{L^3}^2+C\|u\|_{L^6}^2\|\nabla B_t\|_{L^3}^2 +C\|\nabla B_t\|_{L^2}^2\\
&\quad +C\|B_t\|_{L^{\infty}}^2\|\nabla u\|_{L^2}^2+C\|B\|_{L^{\infty}}^2\|\nabla u_t\|_{L^2}^2 \\
&\leqslant \frac{1}{2}\|\nabla^2 B_t\|_{L^2}^2+C\left(\|B_{t t}\|_{L^2}^2+\|\nabla u_t\|_{L^2}^2\|\nabla B\|_{H^1}^2+\|\nabla u\|_{L^2}^4\|\nabla B_t\|_{L^2}^2+\|\nabla B_t\|_{L^2}^2\right),
\end{aligned}
$$
which combined with $\eqref{dbnlg}$, $\eqref{dbsj}$, $\eqref{dux}$, $\eqref{tdugj}$, $\eqref{ddb}$-$\eqref{t2mg}$ and $\eqref{bt3g}$ gives $\eqref{dut3}$ and complets the proof of Lemma 3.8.
\end{proof}
\begin{lemma}
Let $(\rho, u, B, P)$ be a smooth solution to $\eqref{MHD}$-$\eqref{ydxw}$ satisfying $\eqref{p1}$. Then there exists a generic positive constant $C$ depending only on $p, q, \bar{\rho}, \underline{\mu}, \bar{\mu}, \nu, \Omega, \|\rho_0\|_{L^{3/2}}$, $\|B_0\|_{L^{p}}$ and $M$ such that for any $0<\delta<1/8$,  $s\triangleq\min\{6,q\}$, one obtain
\begin{equation}
\begin{split}\label{dut2}
&\sup _{t \in[0, T]} \left(t^{\frac{12-3p}{2p}-\delta}\|\nabla u\|_{W^{1,s}}^2+t^{\frac{12-3p}{2p}-\delta}\|P\|_{W^{1,s}}^2+t^{{\frac{6}{p}-\delta}}\|\nabla u_t\|_{L^2}^2\right)\\
&+\int_0^T t^{{\frac{6-p}{p}-\delta}}\left(\|\nabla u_t\|_{L^2 \cap L^s}^2+\|P_t\|_{L^2 \cap L^s}^2+\|(\rho u_t)_t\|_{L^2}^2\right) \mathrm{~d} t \leqslant C(\delta),
\end{split}
\end{equation}
provided $\|\nabla u_0\|_{L^2}+\|\nabla B_0\|_{L^2}\le \varepsilon_0$.
\end{lemma}
\begin{proof}
First, in similar way to $\eqref{nqgj}$ and $\eqref{mmdq}$, we have
\begin{align}\label{mddg}
\sup _{t \in[0, T]}\|\nabla \rho\|_{L^2} \leqslant 2\|\nabla \rho_0\|_{L^2},
\end{align}
which together with $\eqref{MHD}_1$ and $\eqref{mddg}$ gives
\begin{align}\label{mdtg}
\|\rho_t\|_{L^{3 / 2} \cap L^2} &=\|u\cdot\nabla\rho\|_{L^{3 / 2} \cap L^2}\nonumber\\
&\leqslant C\|\nabla \rho\|_{L^2}\|\nabla u\|_{L^2}^{1 / 2}\|\nabla u\|_{H^1}^{1 / 2} \leqslant C\|\nabla u\|_{L^2}^{1 / 2}\|\nabla u\|_{H^1}^{1/2} .
\end{align}

Next, it follows from $\eqref{tmiu}$ that $u_t$ satisfies
$$
\left\{\begin{array}{l}
-\operatorname{div}(2 \mu(\rho) d_t)+\nabla P_t=G+\operatorname{div} g, \\
\operatorname{div} u_t=0,
\end{array}\right.
$$
with
$$
G \triangleq-\rho u_{t t}-\rho u \cdot \nabla u_t-\rho_t u_t-(\rho u)_t \cdot \nabla u, \quad g \triangleq-2 u \cdot \nabla \mu(\rho) d+(B\otimes B)_t .
$$
Thus, due to Lemma 2.6, one can deduce that
\begin{align}\label{dutsg}
\|\nabla u_t\|_{L^2 \cap L^s}+\|P_t\|_{L^2 \cap L^s} \leqslant C\|g\|_{L^2 \cap L^s}+C\|G\|_{L^{3 s /(3+s)} \cap L^{6 / 5}}, \quad s \triangleq \min \{6, q\} .
\end{align}
Using $\eqref{p1}$, $\eqref{midu}$, $\eqref{dux}$ and $\eqref{mdtg}$, we get by direct calculations that
\begin{equation}\label{gsg}
\begin{split}
& \|G\|_{L^{3 s /(3+s)} \cap L^{6 / 5}} \\
& \leqslant C\|\rho\|_{L^{3 s /(6-s)} \cap L^{3 / 2}}^{1 / 2}\|\rho^{1 / 2} u_{t t}\|_{L^2}+C\|\rho\|_{L^{6 s /(6-s)} \cap L^3}\|u\|_{L^{\infty}}\|\nabla u_t\|_{L^2} \\
&\quad +C\|\rho_t\|_{L^{3 / 2} \cap L^2}\left(\|u_t\|_{L^{6 s /(6-s)} \cap L^6}+\|\nabla u\|_{H^1}^{3/2}+\|\nabla u\|_{H^1}^{1/2}\|\nabla u\|_{W^{1, s}}\right) \\
&\quad +C\|\rho\|_{L^2 \cap L^s}\|u_t\|_{L^6}\|\nabla u\|_{L^6} \\
& \leqslant C\|\rho^{1 / 2} u_{t t}\|_{L^2}+\delta\|\nabla u_t\|_{L^s}+C_\delta\|\nabla u_t\|_{L^2}\|\nabla u\|_{H^1}\left(1+\|\nabla u\|_{H^1}\right)+C\|\nabla u\|_{H^1}^2\\
&\quad+C\|\nabla u\|_{H^1}\|\nabla B\|_{H^1}^{3/2} \text {, }
\end{split}
\end{equation}
where we used the following fact
\begin{gather}
\|u_t\|_{L^\frac{6s}{6-s}}\leqslant C\|u_t\|_{L^6}^\frac{s}{3s-6}\|\nabla u_t\|_{L^s}^\frac{2s-6}{3s-6},\\
\|\nabla u\|_{W^{1,s}}+\|P\|_{W^{1,s}} \leqslant C\left(\|\nabla u_t\|_{L^2}+\|\nabla u\|_{L^2}^3+\|\nabla B\|_{H^1}^{3/2}\right),\label{sugj}
\end{gather}
similar to $\eqref{sugj}$, due to the Sobolev inequality and $\eqref{duwqg}$, one has
\begin{align}\label{sdwq}
\|\nabla u\|_{L^{\infty}} \leqslant C\|\nabla u\|_{H^1 \cap W^{1, s}} \leqslant C\left(\|\nabla u_t\|_{L^2}+\|\nabla u\|_{L^2}^3+\|\nabla B\|_{H^1}^{3/2}\right),
\end{align}
thus
\begin{align}\label{gsgj}
\|g\|_{L^s \cap L^2} & \leqslant C\|\nabla \mu(\rho)\|_{L^q}\|u\|_{L^6 \cap L^{\infty}}\|\nabla u\|_{L^2 \cap L^{\infty}}+C\|B\|_{L^{\infty}}\|B_t\|_{H^1} \nonumber\\
& \leqslant C\|\nabla u_t\|_{L^2}\|\nabla u\|_{H^1}+C\|\nabla u\|_{H^1}\|\nabla B\|_{H^1}^{3/2}+C\|\nabla u\|_{H^1}^2\\
&\quad+C\|\nabla B\|_{H^1}\|B_t\|_{H^1} .\nonumber
\end{align}
Putting $\eqref{gsg}$ and $\eqref{gsgj}$ into $\eqref{dutsg}$, we obtain after choosing $\delta>0$ suitable small that
\begin{align}\label{dutqb}
&\|\nabla u_t\|_{L^2 \cap L^s}+\|P_t\|_{L^2 \cap L^s} \nonumber\\
&\leqslant C\|\rho^{1 / 2} u_{t t}\|_{L^2}+C\|\nabla u_t\|_{L^2}\|\nabla u\|_{H^1}\left(1+\|\nabla u\|_{H^1}\right)+C\|\nabla u\|_{H^1}^2 \\
&\quad+C\|\nabla u\|_{H^1}\|\nabla B\|_{H^1}^{3/2}+C\|\nabla B\|_{H^1}\|B_t\|_{H^1}.\nonumber
\end{align}

Now, multiplying $\eqref{tmiu}$ by $u_{t t}$ and integrating the resulting equation by parts lead to
\begin{align}\label{zht}
\frac{\mathrm{d}}{\mathrm{d} t} & \int \mu(\rho)|d_t|^2 \mathrm{~d} x+\int \rho|u_{t t}|^2 \mathrm{~d} x \nonumber\\
= & \int \operatorname{div}\left(\mu(\rho)\cdot u\right)|d_t|^2 \mathrm{~d} x-\int \rho\left(u \cdot \nabla u_t+u_t \cdot \nabla u\right) \cdot u_{t t} \mathrm{~d} x-\int \rho_t u_t \cdot u_{t t} \mathrm{~d} x \nonumber\\
& -\int \rho_t u \cdot \nabla u \cdot u_{t t} \mathrm{~d} x-2 \int \operatorname{div}\left(u \cdot \nabla \mu(\rho) d\right) \cdot u_{t t} \mathrm{~d} x+\int\left(B_t \cdot \nabla B+B \cdot \nabla B_t\right) \cdot u_{t t} \mathrm{~d} x \nonumber\\
\triangleq& \int \operatorname{div}\left(\mu(\rho)\cdot u\right)|d_t|^2 \mathrm{~d} x+\sum_{i=1}^5 K_i .
\end{align}

We will use $\eqref{xy}$, $\eqref{dbnlg}$, $\eqref{dux}$, $\eqref{dutqb}$ and the Sobolev inequality to estimate each term on the right-hand side of $\eqref{zht}$ as follows:

First, H\"older's inequality gives
\begin{align}
K_1 \leqslant \tau\|\rho^{1 / 2} u_{t t}\|_{L^2}^2+C_\tau\|\nabla u_t\|_{L^2}^2\|\nabla u\|_{H^1}^2.
\end{align}

Next, the inequalities $\eqref{mddg}$, $\eqref{mdtg}$ and $\eqref{dutqb}$ imply that
\begin{align}
K_2= & -\frac{1}{2} \frac{\mathrm{d}}{\mathrm{d} t} \int \rho_t|u_t|^2 \mathrm{~d} x+\int \rho_{t t}|u_t|^2 \mathrm{~d} x \nonumber\\
= & -\frac{\mathrm{d}}{\mathrm{d} t} \int \rho u \cdot \nabla u_t \cdot u_t \mathrm{~d} x+\int(\rho u)_t \cdot \nabla(|u_t|^2) \mathrm{d} x \nonumber\\
\leqslant & -\frac{\mathrm{d}}{\mathrm{d} t} \int \rho u \cdot \nabla u_t \cdot u_t \mathrm{~d} x+C\|\rho_t\|_{L^2}\|u\|_{L^{\infty}}\|\nabla u_t\|_{L^3}\|u_t\|_{L^6} \nonumber\\
& +C\|\rho\|_{L^6}\|u_t\|_{L^6}\|\nabla u_t\|_{L^2}\|u_t\|_{L^6} \nonumber\\
\leqslant & -\frac{\mathrm{d}}{\mathrm{d} t} \int \rho u \cdot \nabla u_t \cdot u_t \mathrm{~d} x+\tau\|\rho^{1 / 2} u_{t t}\|_{L^2}^2+C_\tau\|\nabla u_t\|_{L^2}^2\left(\|\nabla u\|_{H^1}^2+\|\nabla u\|_{H^1}^4\right)+C_\tau\|\nabla u\|_{H^1}^6 \nonumber\\
& +C_\tau\|\nabla u\|_{H^1}\|\nabla u_t\|_{L^2}^3+C_\tau\|\nabla u\|_{H^1}^2\|\nabla u_t\|_{H^1}\left(\|\nabla u\|_{H^1}\|\nabla B\|_{H^1}^{3/2}+\|\nabla B\|_{H^1}\|B_t\|_{H^1}\right).\nonumber
\end{align}

Then, it follows from $\eqref{MHD}_1$ and $\eqref{mdtg}$ that
\begin{align}
K_3= & -\frac{\mathrm{d}}{\mathrm{d} t} \int \rho_t u \cdot \nabla u \cdot u_t \mathrm{~d} x-\int \operatorname{div}(\rho u)_t u \cdot \nabla u \cdot u_t \mathrm{~d} x+\int \rho_t(u \cdot \nabla u)_t \cdot u_t \mathrm{~d} x \nonumber\\
= & -\frac{\mathrm{d}}{\mathrm{d} t} \int \rho_t u \cdot \nabla u \cdot u_t \mathrm{~d} x+\int \rho u_t^i\left(\partial_i(u^j \partial_j u^k) u_t^k+u^j \partial_j u^k \partial_i u_t^k\right) \mathrm{d} x \nonumber\\
& +\int \rho_t u^i\left(\partial_i(u^j \partial_j u^k) u_t^k+u^j \partial_j u^k \partial_i u_t^k\right) \mathrm{d} x+\int \rho_t\left(u_t^i \partial_i u^j+u^i \partial_i u_t^j\right) u_t^j \mathrm{~d} x \nonumber\\
\leqslant & -\frac{\mathrm{d}}{\mathrm{d} t} \int \rho_t u \cdot \nabla u \cdot u_t \mathrm{~d} x+C\|u_t\|_{L^6}\left(\|\nabla u\|_{L^3}^2+\|u\|_{L^6}\|\nabla^2 u\|_{L^2}\right)\|u_t\|_{L^6} \nonumber\\
& +C\|u_t\|_{L^6}\|u\|_{L^6}\|\nabla u\|_{L^6}\|\nabla u_t\|_{L^2}+C\|\rho_t\|_{L^2}\|u\|_{L^{\infty}}\left(\|\nabla u\|_{L^6}^2+\|u\|_{L^{\infty}}\|\nabla^2 u\|_{L^3}\right)\|u_t\|_{L^6} \nonumber\\
& +C\|\rho_t\|_{L^2}\|u\|_{L^{\infty}}^2\|\nabla u\|_{L^{\infty}}\|\nabla u_t\|_{L^2}+C\|\rho_t\|_{L^2}\left(\|u_t\|_{L^6}\|\nabla u\|_{L^6}+\|u\|_{L^{\infty}}\|\nabla u_t\|_{L^3}\right)\|u_t\|_{L^6} \nonumber\\
\leqslant & -\frac{\mathrm{d}}{\mathrm{d} t} \int \rho_t u \cdot \nabla u \cdot u_t \mathrm{~d} x+\tau\|\rho^{1 / 2} u_{t t}\|_{L^2}^2+C_\tau\|\nabla u_t\|_{L^2}^2\left(\|\nabla u\|_{H^1}^2+\|\nabla u_t\|_{L^2}+\|\nabla u\|_{H^1}^4\right) \nonumber\\
& +C_\tau\|\nabla u\|_{H^1}^6+C_\tau\|\nabla u\|_{H^1}^2\|\nabla u_t\|_{H^1}\left(\|\nabla u\|_{H^1}\|\nabla B\|_{H^1}^{3/2}+\|\nabla B\|_{H^1}\|B_t\|_{H^1}\right),\nonumber
\end{align}
where in the last inequality, one has used $\eqref{mdtg}$ and $\eqref{dutqb}$.

Next, direct calculations lead to
$$
\begin{aligned}
K_4= & -2 \frac{\mathrm{d}}{\mathrm{d} t} \int \partial_i\left(u^j \partial_j \mu(\rho) d_i^k\right) u_t^k \mathrm{~d} x+2 \int \partial_i\left(u^j \partial_j \mu(\rho) d_i^k\right)_t u_t^k \mathrm{~d} x \\
= & 2 \frac{\mathrm{d}}{\mathrm{d} t} \int u^j \partial_j \mu(\rho) d_i^k \partial_i u_t^k \mathrm{~d} x+2 \int\left(\mu(\rho) u^j \partial_j d_i^k\right)_t \partial_i u_t^k \mathrm{~d} x-2 \int\left(u^j \partial_j(\mu(\rho) d_i^k)\right)_t \partial_i u_t^k \mathrm{~d} x \\
= & 2 \frac{\mathrm{d}}{\mathrm{d} t} \int u^j \partial_j \mu(\rho) d_i^k \partial_i u_t^k \mathrm{~d} x+2 \int\left(\mu(\rho) u^j \partial_j d_i^k\right)_t \partial_i u_t^k \mathrm{~d} x-2 \int\left(\partial_i u^j \mu(\rho) d_i^k\right)_t \partial_j u_t^k \mathrm{~d} x \\
& -2 \int u_t^j \partial_i\left(\mu(\rho) d_i^k\right) \partial_j u_t^k \mathrm{~d} x-2 \int u^j\left(\partial_i(\mu(\rho) d_i^k)\right)_t \partial_j u_t^k \mathrm{~d} x \\
\triangleq & 2 \frac{\mathrm{d}}{\mathrm{d} t} \int u^j \partial_j \mu(\rho) d_i^k \partial_i u_t^k \mathrm{~d} x+\sum_{l=1}^4 K_{4, l} .
\end{aligned}
$$
We estimate each $K_{4, l}(l=1, \ldots, 4)$ as follows:

First, integration by parts gives
$$
\begin{aligned}
K_{4,1}&=2 \int(\mu(\rho) u^j)_t \partial_j d_i^k \partial_i u_t^k \mathrm{~d} x+2 \int \mu(\rho) u^j \partial_j(d_i^k)_t \partial_i u_t^k \mathrm{~d} x \\
& =2 \int\left(\mu(\rho) u_t^j-u \cdot \nabla \mu(\rho) u^j\right) \partial_j d_i^k \partial_i u_t^k \mathrm{~d} x-\int \operatorname{div}(\mu(\rho) u)|d_t|^2 \mathrm{~d} x \\
& \leqslant C\|u_t\|_{L^6}\|\nabla^2 u\|_{L^3}\|\nabla u_t\|_{L^2}+C\|u\|_{L^{6 q /(q-3)}}^2\|\nabla \mu(\rho)\|_{L^q}\|\nabla^2 u\|_{L^3}\|\nabla u_t\|_{L^3} \\
&\quad -\int \operatorname{div}(\mu(\rho) u)|d_t|^2 \mathrm{~d} x \\
& \leqslant-\int \operatorname{div}(\mu(\rho) u)|d_t|^2 \mathrm{~d} x+\frac{\tau}{2}\|\rho^{1 / 2} u_{t t}\|_{L^2}^2+C_\tau\|\nabla u_t\|_{L^2}^2\left(\|\nabla u\|_{H^1}^2+\|\nabla u\|_{H^1}^4+\|\nabla u_t\|_{L^2}\right) \\
&\quad +C_\tau\left(\|\nabla u_t\|_{L^2}^2\|\nabla B\|_{H^1}^{3/2}+\|\nabla u\|_{H^1}^6+\|\nabla u\|_{H^1}^4\|\nabla B\|_{H^1}^{3/2}+\|\nabla u\|_{H^1}^3\|\nabla B\|_{H^1}^{3}\right) \text {, } \\
&\quad+C_\tau\|\nabla u\|_{H^1}^2\left(\|\nabla u_t\|_{L^2}+\|\nabla u\|_{H^1}^{3/2}+\|\nabla B\|_{H^1}^{3/2}\right)\|\nabla B\|_{H^1}\|B_t\|_{H^1},
\end{aligned}
$$
where in the last inequality we have used $\eqref{dux}$, $\eqref{sugj}$ and $\eqref{dutqb}$.

Then, it follows from $\eqref{p1}$ and $\eqref{sdwq}$ that
$$
\begin{aligned}
K_{4,2} & =-2 \int \partial_i u_t^j \mu(\rho) d_i^k \partial_j u_t^k \mathrm{~d} x-2 \int \partial_i u^j(\mu(\rho))_t d_i^k \partial_j u_t^k \mathrm{~d} x-2 \int \partial_i u^j \mu(\rho)(d_i^k)_t \partial_j u_t^k \mathrm{~d} x \\
& \leqslant C\|\nabla u\|_{L^{\infty}}\|\nabla u_t\|_{L^2}^2+C\|u\|_{L^{\infty}}\|\nabla u\|_{L^{3 q /(q-3)}}\|\nabla \mu(\rho)\|_{L^q}\|\nabla u\|_{L^6}\|\nabla u_t\|_{L^2} \\
& \leqslant C\|\nabla u_t\|_{L^2}^2\left(\|\nabla u_t\|_{L^2}+\|\nabla u\|_{H^1}^2+\|\nabla B\|_{H^1}^{3/2}\right)+C\|\nabla u_t\|_{L^2}\|\nabla u\|_{H^1}^3.
\end{aligned}
$$
Similarly, combining H\"older's inequality and $\eqref{sugj}$ leads to
$$
\begin{aligned}
K_{4,3} & =-2 \int u_t^j \partial_i \mu(\rho) d_i^k \partial_j u_t^k \mathrm{~d} x-2 \int u_t^j \mu(\rho) \partial_i d_i^k \partial_j u_t^k \mathrm{~d} x \\
& \leqslant C\|u_t\|_{L^6}\|\nabla \mu(\rho)\|_{L^q}\|\nabla u\|_{L^{3 q /(q-3)}}\|\nabla u_t\|_{L^2}+C\|u_t\|_{L^6}\|\nabla^2 u\|_{L^3}\|\nabla u_t\|_{L^2} \\
& \leqslant C\|\nabla u_t\|_{L^2}^2\left(\|\nabla u_t\|_{L^2}+\|\nabla u\|_{H^1}+\|\nabla u\|_{H^1}^2+\|\nabla B\|_{H^1}^\frac{3}{2}\right) .
\end{aligned}
$$

Finally, using $\eqref{MHD}_2$ and $\eqref{MHD}_3$, we obtain after integrating by parts that
$$
\begin{aligned}
K_{4,4}= & -\int u^j\left(\rho u_t^k-\rho u \cdot \nabla u^k\right)_t \partial_j u_t^k \mathrm{~d} x-\int u^j \partial_k P_t \partial_j u_t^k \mathrm{~d} x+\int u^j(B \cdot \nabla B)_t \partial_j u_t^k \mathrm{~d} x \\
\leqslant & C\|u\|_{L^{\infty}}\|\rho_t\|_{L^2}\|u_t\|_{L^{\infty}}\|\nabla u_t\|_{L^2}+C\|u\|_{L^{\infty}} \| \rho^{1 / 2} u_{t t}\|_{L^2}\| \nabla u_t \|_{L^2} \\
& +C\|u\|_{L^{\infty}}^2\|\rho_t\|_{L^2}\|\nabla u\|_{L^{\infty}}\|\nabla u_t\|_{L^2}+C\|u\|_{L^{\infty}} C\|u_t\|_{L^6}\|\nabla u\|_{L^3}\|\nabla u_t\|_{L^2} \\
& +C\|u\|_{L^{\infty}}^2\|\nabla u_t\|_{L^2}^2+C\|\nabla u\|_{L^6}\|P_t\|_{L^3}\|\nabla u_t\|_{L^2} \\
& +C\|u\|_{L^{\infty}}\left(\|B_t\|_{L^6}\|\nabla B\|_{L^3}+\|B\|_{L^{\infty}}\|\nabla B_t\|_{L^2}\right)\|\nabla u_t\|_{L^2} \\
\leqslant & \frac{\tau}{2}\|\rho^{1 / 2} u_{t t}\|_{L^2}^2+C_\tau\|\nabla u_t\|_{L^2}^2\left(\|\nabla u\|_{H^1}+\|\nabla u\|_{H^1}^4+\|\nabla u_t\|_{L^2}\right) \\
& +C_\tau\left(\|\nabla u_t\|_{L^2}\|\nabla u\|_{H^1}^3+\|\nabla u\|_{H^1}\|\nabla u_t\|_{L^2}\|\nabla B\|_{H^1}\|B_t\|_{H^1}+\|\nabla u\|_{H^1}^6\right)\\
&+C_\tau\left(\|\nabla u\|_{H^1}^2\|\nabla B\|_{H^1}^{3/2}\|\nabla u_t\|_{L^2}+\|\nabla u\|_{H^1}\|\nabla u_t\|_{L^2}\|\nabla B\|_{H^1}\|B_t\|_{H^1}\right),
\end{aligned}
$$
where in the last inequality one has used $\eqref{mdtg}$ and $\eqref{dutqb}$.

Substituting $K_{4,1}-K_{4,4}$ into $K_4$, we get after choosing $\tau$ suitably small that
\begin{align}
K_4 \leqslant & 2 \frac{\mathrm{d}}{\mathrm{d} t} \int u^j \partial_j \mu(\rho) d_i^k \partial_i u_t^k \mathrm{~d} x-\int \operatorname{div}(\mu(\rho) u)|d_t|^2 \mathrm{~d} x+\tau\|\rho^{1 / 2} u_{t t}\|_{L^2}^2 \nonumber\\
&+C_\tau\|\nabla u_t\|_{L^2}^2\left(\|\nabla u\|_{H^1}+\|\nabla u\|_{H^1}^4+\|\nabla u_t\|_{L^2}+\|\nabla B\|_{H^1}^{3/2}\right)+C_\tau\|\nabla u\|_{H^1}^6 \nonumber\\
&+C_\tau\|\nabla u\|_{H^1}^4\|\nabla B\|_{H^1}^{3/2}+C_\tau\|\nabla u_t\|_{L^2}\|\nabla u\|_{H^1}\left(\|\nabla u\|_{H^1}^2+\|\nabla B\|_{H^1}\|B_t\|_{H^1}\right) \nonumber\\
&+C_\tau\|\nabla u\|_{H^1}^3\|\nabla B\|_{H^1}^{3}+C_\tau\|\nabla u\|_{H^1}^2\left(\|\nabla u_t\|_{L^2}+\|\nabla u\|_{H^1}^{3/2}+\|\nabla B\|_{H^1}^{3/2}\right)\|\nabla B\|_{H^1}\|B_t\|_{H^1}\nonumber\\
&+C_\tau\left(\|\nabla u\|_{H^1}^2\|\nabla B\|_{H^1}^{3/2}\|\nabla u_t\|_{L^2}+\|\nabla u\|_{H^1}\|\nabla u_t\|_{L^2}\|\nabla B\|_{H^1}\|B_t\|_{H^1}\right).\nonumber
\end{align}
For $K_5$, the H\"older's inequality gives
\begin{align}
K_5= & \frac{\mathrm{d}}{\mathrm{d} t} \int(B \cdot \nabla B)_t \cdot u_t \mathrm{~d} x-\int\left(B_{t t} \cdot \nabla B+2 B_t \cdot \nabla B_t+B \cdot \nabla B_{t t}\right) \cdot u_t \mathrm{~d} x \nonumber\\
\leqslant & \frac{\mathrm{d}}{\mathrm{d} t} \int(B \cdot \nabla B)_t \cdot u_t \mathrm{~d} x+C\|u_t\|_{L^6}\left(\|B_{t t}\|_{L^2}\|\nabla B\|_{L^3}+\|B_t\|_{L^3}\|\nabla B_t\|_{L^2}\right) \nonumber\\
& +C\|B\|_{L^{\infty}}\|\nabla u_t\|_{L^2}\|B_{t t}\|_{L^2} \nonumber\\
\leqslant & \frac{\mathrm{d}}{\mathrm{d} t} \int(B \cdot \nabla B)_t \cdot u_t \mathrm{~d} x+C\|\nabla u_t\|_{L^2}\|B_{tt}\|_{L^2}\|\nabla B\|_{H^1}+C\|\nabla u_t\|_{L^2}\|B_t\|_{L^2}^{1/2}\|\nabla B_t\|_{L^2}^{3/2}.\nonumber
\end{align}
Substituting $K_1$-$K_5$ into $\eqref{zht}$, we get after choosing $\tau>0$ suitable small that
\begin{align}\label{dthbz}
& \frac{\mathrm{d}}{\mathrm{d} t} \int \mu(\rho)|d_t|^2 \mathrm{~d} x+\varphi^{\prime}(t)+\int \rho|u_{t t}|^2 \mathrm{~d} x \nonumber\\
& \leqslant C\|\nabla u_t\|_{L^2}^2\left(\|\nabla u\|_{H^1}+\|\nabla u\|_{H^1}^4+\|\nabla u_t\|_{L^2}+\|\nabla B\|_{H^1}^{3/2}\right)+C\|\nabla u\|_{H^1}^6 \nonumber\\
&\quad+C\|\nabla u\|_{H^1}\|\nabla u_t\|_{L^2}^3+C\|\nabla u\|_{H^1}^2\|\nabla u_t\|_{L^2}\left(\|\nabla u\|_{H^1}\|\nabla B\|_{H^1}^{3/2}+\|\nabla B\|_{H^1}\|B_t\|_{H^1}\right)\nonumber\\
&\quad+C\|\nabla u\|_{H^1}^4\|\nabla B\|_{H^1}^{3/2}+C\|\nabla u_t\|_{L^2}\|\nabla u\|_{H^1}\left(\|\nabla u\|_{H^1}^2+\|\nabla B\|_{H^1}\|B_t\|_{H^1}\right) \\
&\quad+C\|\nabla u\|_{H^1}^3\|\nabla B\|_{H^1}^{3}+C\|\nabla u\|_{H^1}^2\left(\|\nabla u_t\|_{L^2}+\|\nabla u\|_{H^1}^{3/2}+\|\nabla B\|_{H^1}^{3/2}\right)\|\nabla B\|_{H^1}\|B_t\|_{H^1}\nonumber\\
&\quad+C\|\nabla u\|_{H^1}^2\|\nabla B\|_{H^1}^{3/2}\|\nabla u_t\|_{L^2}+C\|\nabla u\|_{H^1}\|\nabla u_t\|_{L^2}\|\nabla B\|_{H^1}\|B_t\|_{H^1}\nonumber\\
&\quad+C\|\nabla u_t\|_{L^2}\|B_{tt}\|_{L^2}\|\nabla B\|_{H^1}+C\|\nabla u_t\|_{L^2}\|B_t\|_{L^2}^{1/2}\|\nabla B_t\|_{L^2}^{3/2},\nonumber
\end{align}
where
$$
\begin{aligned}
\varphi(t) \triangleq & -\int \rho u \cdot \nabla u_t \cdot u_t-\int \rho_t u \cdot \nabla u \cdot u_t \mathrm{~d} x +2 \int u^j \partial_j \mu(\rho) d_i^k \partial_i u_t^k \mathrm{~d} x\\&+\int(B \cdot \nabla B)_t \cdot u_t \mathrm{~d} x,
\end{aligned}
$$
satisfies
\begin{align}\label{fgj}
|\varphi(t)| &\leqslant C\|u\|_{L^{\infty}}\|\nabla u_t\|_{L^2}\|\rho^{1 / 2} u_t\|_{L^2}+C\|\rho_t\|_{L^2}\|u\|_{L^6}\|\nabla u\|_{L^6}\|u_t\|_{L^6} \nonumber\\
&\quad +C\|u\|_{L^{3 q /(q-3)}}\|\nabla \mu(\rho)\|_{L^q}\|d\|_{L^6}\|\nabla u_t\|_{L^2}+C\|B\|_{L^{\infty}}\|B_t\|_{L^2}\|\nabla u_t\|_{L^2} \\
& \leqslant \tau\|\nabla u_t\|_{L^2}^2+C_\tau\left(\|\nabla u\|_{H^1}^2\|\rho^{1 / 2} u_t\|_{L^2}^2+\|\nabla u\|_{H^1}^4+\|\nabla B\|_{H^1}^2\|B_t\|_{L^2}^2\right) ,\nonumber
\end{align}
due to $\eqref{p1}$ and $\eqref{mdtg}$.

Then, multiplying $\eqref{dthbz}$ by $t^{\frac{6}{p}-\delta}$, and using $\eqref{p0}$, Lemma 3.2-3.8 and $\eqref{du5g}$, we get
\begin{align}
\sup _{0 \leq t \leq T} t^{\frac{6}{p}-\delta}\|\nabla u_t\|_{L^2}^2+\int_0^T t^{\frac{6}{p}-\delta}\|\rho^{1 / 2} u_{t t}\|_{L^2}^2 dt \leqslant C(\delta) .\label{7ut}
\end{align}
In particular, multiplying $\eqref{sugj}$ and $\eqref{dutqb}$ by $t^{\frac{12-3p}{2p}-\delta}$ and $t^{\frac{6-p}{p}-\delta}$, by \eqref{dbsj}$, \eqref{t2mg}$ and $\eqref{du5g}$ , we get
\begin{align}
\sup _{t \in[0, T]} t^{{\frac{12-3p}{2p}-\delta}}\left(\|\nabla u\|_{W^{1,s}}^2+\|P\|_{W^{1,s}}^2\right)+\int_0^T t^{\frac{6-p}{p}-\delta}\left(\|\nabla u_t\|_{L^2\cap L^s}^2 + \|P_t\|_{L^2\cap L^s}^2\right)dt \leqslant C(\delta) .\label{duszj}
\end{align}
Furthermore, noticing that by $\eqref{mdtg}$
$$
\|(\rho u_t)_t\|_{L^2}^2 \leqslant C\|\nabla u\|_{H^1}^2\|\nabla u_t\|_{L^2 \cap L^s}^2+C\|\rho^{1 / 2} u_{t t}\|_{L^2}^2,
$$
which together with $\eqref{du5g}$, $\eqref{7ut}$ and $\eqref{duszj}$, we derive $\eqref{dut2}$ and thus completes the proof of Lemma 3.9.
\end{proof}
\section{Proof of Theorem 1.1-1.2}
With all the a priori estimates in Section 3 at hand, we are now in a position to prove Theorems 1.1 and 1.2.\\
\\
\textbf{Proof of Theorem 1.1.}
First, by Lemma 2.1, there exists a $T_*>0$ such that the system $\eqref{MHD}$-$\eqref{ydxw}$ has a unique local strong solution $(\rho, u, B, P)$ in $\Omega \times(0, T_*]$. By $\eqref{cztj}$, there exists a $T_1 \in(0, T_*]$ such that $\eqref{p1}$ holds for $T=T_1$.

Next, set
\begin{align}\label{tcsq}
T^* \triangleq \sup \{T\mid(\rho, u, B, P)\text{\ is\ a\ strong\ solution\ on}\ \Omega\times(0,T]\ \text{and}\ (3.1) \text{ holds}\}.
\end{align}
Then $T^* \geq T_1>0$. Hence, it follow from Lemma 3.2-3.9 that for any $0<\tau<T \leq T^*$ with $T$ finite, one obtain
\begin{align}\label{Tlx}
\nabla u, B, P \in C([\tau, T] ; L^2) \cap C(\bar{\Omega} \times[\tau, T]),\quad \nabla B \in C([\tau, T] ; H^1),
\end{align}
where we have utilized the following standard embedding
$$
\begin{aligned}
&H^1(\tau, T ; L^2) \hookrightarrow C([\tau, T] ; L^2) ,\\
&L^{\infty}(\tau, T ; H^1\cap W^{1,s}) \cap H^1(\tau, T ; L^2) \hookrightarrow C(\bar{\Omega} \times[\tau, T]).
\end{aligned}
$$
Moreover, it from $\eqref{p1}$, $\eqref{midu}$, $\eqref{mddg}$, $\eqref{mdtg}$ and \cite[Lemma 2.3]{Lio1996} that
\begin{align}\label{mdlx}
\rho \in C([0, T] ; L^{3 / 2} \cap H^1), \quad \nabla \mu(\rho) \in C([0, T] ; L^q).
\end{align}
Due to $\eqref{dux}$ and $\eqref{dut2}$, the standard arguments leads to
$$
\rho u_t \in H^1(\tau, T ; L^2) \hookrightarrow C([\tau, T] ; L^2),
$$
which together with $\eqref{Tlx}$ and $\eqref{mdlx}$ implies
\begin{align}\label{fqlx}
\rho u_t+\rho u \cdot \nabla u -B\cdot\nabla B \in C([\tau, T] ; L^2).
\end{align}
On the other hand, noting $(\rho, u)$ satisfies $\eqref{sto}$ with $F \equiv \rho u_t+\rho u \cdot \nabla u-B\cdot\nabla B$, and then by $\eqref{MHD}$, $\eqref{Tlx}$, $\eqref{mdlx}$, $\eqref{fqlx}$ and $\eqref{dut2}$, we find that for any $\gamma \in[2, s)$, $s\triangleq\min\{6,q\}$, one have
\begin{align}\label{dulx}
\nabla u, P \in C([\tau, T] ; D^1 \cap D^{1, \gamma}) .
\end{align}

Now, we claim that
\begin{align}\label{twql}
T^*=\infty.
\end{align}
Otherwise, $T^*<\infty$. By Proposition 3.1, it indicates that $\eqref{p2}$ holds at $T=T^*$. By virtue of $\eqref{Tlx}$, $\eqref{mdlx}$ and $\eqref{dulx}$, one could set
$$
(\rho^*, u^*, B^*)(x) \triangleq(\rho, u, B)(x, T^*)=\lim _{t \rightarrow T^*}(\rho, u, B)(x, t),
$$
and then
$$
\rho^* \in L^{3 / 2} \cap H^1, \quad u^* \in D_{0, \sigma}^1\cap D^{1,\gamma}, \quad B^* \in H_{0, \sigma}^1
$$
for any $\gamma \in[2, s)$. Consequently, $(\rho^*, \rho^* u^*, B^*)$ could be taken as the initial data and Lemma 2.1 implies that there exists some $T^{* *}>T^*$ such that $\eqref{p1}$ holds for $T=T^{* *}$, which contradicts the definition of $T^*$. So $\eqref{twql}$ holds. $\eqref{dmidj}$ and $\eqref{dutgui}$ come directly from $\eqref{usj}$, $\eqref{dbsj}$, $\eqref{t2mg}$, $\eqref{du5g}$, $\eqref{dut3}$ and $\eqref{mddg}$. We complete the proof of Theorem.

\textbf{Proof of Theorem 1.2}: With the global existence result at hand (see Proposition 1.1), one can modify slightly the proofs of Lemma 3.2-3.9 and $\eqref{mddg}$ to obtain $\eqref{cmgj}$ and $\eqref{gdutgu}$.

\end{document}